\documentclass[reqno]{amsart}

\usepackage{mathrsfs}
\usepackage{amsfonts}
\usepackage{amssymb}
\usepackage{xcolor}
\usepackage{caption}

\usepackage{algorithm}
\usepackage{algorithmic}

\usepackage[colorlinks,citecolor = red,
linkcolor=blue,hyperindex,CJKbookmarks]{hyperref}
\usepackage{cite}
\usepackage{graphicx}
\usepackage{multirow}

\usepackage{float}
\usepackage{subcaption}

\newtheorem{theorem}{Theorem}[section]
\newtheorem{lemma}[theorem]{Lemma}

\newtheorem{assumption}[theorem]{Assumption}

\theoremstyle{definition}

\newcommand{\ba}{\begin{array}}
	\newcommand{\ea}{\end{array}}
\newcommand{\bea}{\begin{eqnarray}}
	\newcommand{\eea}{\end{eqnarray}}

\renewcommand{\P}{\mathbb{P}}

\newcommand{\R}{{\mathbb{R}}}

\numberwithin{equation}{section}

\allowdisplaybreaks

\begin{document}
	
\title[A derivative-free regularization algorithm for ECNLS]
{A derivative-free regularization algorithm for equality constrained nonlinear least squares problems}

\author{Xi Chen}
\address{School of Mathematical Sciences, Shanghai Jiao Tong University, Shanghai 200240, China}
\email{chenxi\_1998@sjtu.edu.cn}
	
\author{Jinyan Fan*}
\address{School of Mathematical Sciences, and MOE LSC, Shanghai Jiao Tong University, Shanghai 200240,		China}
\email{jyfan@sjtu.edu.cn}
	
\thanks{* Corresponding author}
\thanks{The authors are supported by the Science and Technology Commission of Shanghai Municipality grant 22JC1401500, the National Natural Science Foundation of China grant 12371307 and the Fundamental Research Funds for the Central Universities.}
	
	
	
\maketitle
	
\begin{abstract}
In this paper, we study the equality constrained nonlinear least squares problem, where the Jacobian matrices of the objective function and constraints are unavailable or expensive to compute.
We approximate the Jacobian matrices via orthogonal spherical smoothing and propose a derivative-free regularization algorithm for solving the problem.
At each iteration, a regularized augmented Lagrangian subproblem is solved to obtain a Newton-like step. If a sufficient decrease in the merit function of the approximate KKT system is achieved, the step is accepted, otherwise a derivative-free LM algorithm is applied to get another step to satisfy the sufficient decrease condition.
It is shown that the algorithm either finds an approximate KKT point with arbitrary high probability or converges to a stationary point of constraints violation almost surely.

\end{abstract}

\section{Introduction}
Nonlinear least-squares problems have important applications in data fitting, parameter estimation and function approximation. They appear frequently in nuclear physics, aircraft design, aeronautical engineering and so on, where the functions themselves may not be in an analytical closed form, permitting function evaluations but not gradient assessments.
In this paper, we consider the equality constrained nonlinear least squares problem
\begin{align}\label{prob}
	\min_{x\in\mathbb{R}^n} \ \	& \frac{1}{2}\|r(x)\|^2,\nonumber\\
		\text{s.t.} \ \ &c(x)=0,
\end{align}
where  $r(x):\mathbb{R}^n \to \mathbb{R}^p$, $c(x):\mathbb{R}^n \to \mathbb{R}^m$ are continuously differentiable but explicit expressions of their derivatives are difficult to compute or infeasible
to obtain, and $\|\cdot\|$ refers to the standard Euclidian norm.	
Hence traditional nonlinear least-squares methods that use derivative information are not applicable
anymore.
	
It is worth noting that derivative-free methods had been studied by the optimization community long before. They can be classified into the categories: direct search-based methods, model-based methods and the others.
Direct search methods include the Nelder–Mead simplex method \cite{NM65}, coordinate search method \cite{Fermi52}, pattern search method \cite{Torczon91}, mesh adaptive direct search \cite{mads2009,mads,AUDET2020467}, etc.
Gratton et al. \cite{Gratton2019} propose a direct search method based on probabilistic feasible descent for bound and linearly constrained problems.

Powell developed  derivative-free trust-region methods UOBYQA and NEWUOA for unconstrained optimization  \cite{PowellUOBYQA,PowellNEWUOA},
and BOBYQA, LINCOA and COBYLA for bound constrained, linear constrained and general nonlinear constrained optimization, respectively \cite{PowellBOBYQA, PowellLINCOA,PowellCOBYLA},
where BOBYQA and  LINCOA build  quadratic polynomial models by the derivative-free symmetric Broyden update, and  COBYLA constructs linear polynomial models to approximate the objective function and constraints.
Tröltzsch \cite{Troltzsch2016} presented a sequential quadratic programming algorithm for  equality constrained optimization without derivatives \cite{Troltzsch2016}.
Sampaio \cite{Sampaio2021} gave the algorithm DEFT-FUNNEL for general constrained grey-box and black-box problems that belongs to the class of trust-region sequential quadratic optimization algorithms and uses the polynomial interpolation models as surrogates for the black-box functions.
Interested readers are referred to \cite{ AUDET2017, Conn09b, larson} for more model-based derivative-free optimization methods.

Derivative-free methods that use finite differences or smoothing technique to compute estimates of gradients
also attract increasing attention in recent years \cite{FD,FDLM,liuzhong}.
For example, Ge et al. \cite{soln} proposed the derivative-free solver SOLNP+ for constrained nonlinear optimization which uses finite difference to approximate the gradients and incorporates the techniques of implicit filtering, new restart mechanism and modern quadratic programming.
Flaxman et al. \cite{Flax} used spherical smoothing to approximate the gradients
and	Nesterov et al. \cite{YN} designed a kind of random gradient-free oracle by Gaussian smoothing.
Kozak et al. \cite{DavidKozak} approximated the exact gradient by finite differences computed in a set of orthogonal random directions that changes with each iteration.

For unconstrained  nonlinear least squares problems,
Zhang et al. \cite{ZhangHC} built quadratic interpolation models for each $r_i$ and gave a framework for a class of derivative-free algorithms within a trust region.
Cartis et al. \cite{DFOGN} constructed linear interpolation models of all $r_i$ and presented a derivative-free version of the Gauss-Newton method, further linear regression models were used to improve the flexibility and robustness.
Zhao et al. \cite{zhao} built the probabilistically first-order accurate Jacobian models with random points generated by the standard Gaussian distribution and gave a derivative-free Levenberg-Marquardt (LM) algorithm.
Brown et al. \cite{DFOLM} estimated the Jacobian matrices by coordinate-wise finite difference and gave the derivative-free analogues of the LM algorithm and Gauss-Newton algorithm.
Recently, the authors \cite{cx} proposed a derivative-free LM algorithm that approximates the Jacobian matrices via orthogonal spherical smoothing; it is shown that the gradient models with such Jacobians are probabilistically first-order accurate and the algorithm converges globally almost surely.

For constrained nonlinear least squares problems,
Hough et al. \cite{Hough} presented a model-based derivative-free method for optimization subject to unrelaxable convex constraints, where the constraint set is available to the algorithm through a projection operator that is inexpensive to compute.
However, to the best of our knowledge, there are no existing derivative-free methods designed for general constrained nonlinear least squares problems.    	

In this paper, we propose a derivative-free regularization algorithm for equality constrained nonlinear least squares problem \eqref{prob}, where  the Jacobian matrices of the objective function and constraints are unavailable or expensive to compute.
We estimate the Jacobian matrices via orthogonal spherical smoothing.
Based on the framework of the augmented Lagrangian method,
a regularization subproblem is solved at each iteration to obtain a  Newton-like step.
If a sufficient decrease in the merit function of the approximate KKT system is achieved,
the step is accepted, otherwise an improved derivative-free LM algorithm presented in \cite{cx} is used to get a step to satisfy the sufficient decrease condition.
It is shown that the algorithm either finds an approximate KKT point with high probability
or converges to a stationary point of the constraints almost surely.

The paper is organized as follows.
In Section \ref{sec2}, we approximate the Jacobian matrices via orthogonal spherical smoothing
and propose a derivative-free regularization algorithm for \eqref{prob}.
The properties and global convergence of the algorithm are studied in Section \ref{sec3}.
Finally, some numerical experiments are presented in Section \ref{sec4}.

\section{A derivative-free regularization algorithm for \eqref{prob}}\label{sec2}
In this section, we approximate the Jacobian matrices via orthogonal spherical smoothing, then propose a derivative-free regularization for solving the equality constrained nonlinear least squares problem \eqref{prob}.

\subsection{Approximating Jacobian matrices via orthogonal spherical smoothing}	
Define the spherical smoothed version of $r_i(x), i=1,\ldots,p$ as
	\begin{equation}\label{spsmoo}
		\tilde r_{i}(x)=\mathbb{E}_{v\sim U(\mathbb B)} [r_i(x+\gamma v)],
	\end{equation}
	where $\mathbb B$ represents the unit ball in $\R^n$ and  $U(\mathbb B)$ the uniform distribution on  $\mathbb B$,     $\gamma$ is the smoothing parameter.
It then follows from \cite{Flax} that
	\begin{equation}\label{ssmooder}
		\nabla \tilde r_{i} (x)= \frac{n}{\gamma} \mathbb{E}_{u\sim U(\mathbb S)} 	[r_i(x+\gamma u)u],
	\end{equation}
where $\mathbb S$ represents the unit sphere in $\R^n$ and $U(\mathbb S)$ the 	uniform distribution on $\mathbb S$. 	

We take
	\begin{align}\label{nabrrd}
		\nabla r_{s,i}(x)= \sum_{j=1}^n \frac{r_i(x+\gamma 	u_j)-r_i(x)}{\gamma}u_j
	\end{align}
as a gradient approximation of $ r_i(x)$,
where $[u_1,\ldots,u_n]$ is uniformly sampled from the Stiefel manifold $ St(n)=\{U\in \R^{n \times n}: U^TU=I_n\}$, and can be obtained by using the Gram-Schmidt orthogonalization on $n$ independent random vectors with identical distribution 	$\mathcal{N}(0,I)$ (cf. \cite{wang, Chikuse2003}).
	Since the  distribution for such $u_j$ is uniform over $\mathbb S$ (cf.	\cite{DavidKozak}), we have
	\begin{align}
\mathbb{E}_u [\nabla r_{s,i}(x)] =\sum_{j=1}^n \mathbb{E}_u 	\left[\frac{r_i(x+\gamma u_j)-r_i(x)}{\gamma}u_j\right]	=\nabla \tilde r_{i} (x).
\end{align}
We take
\begin{align}\label{JArd}
		J^r_{s}(x)= \left(\nabla r_{s,1}(x), \ldots, \nabla r_{s,p}(x)\right)^T
\end{align}
as the approximate Jacobian matrix of $r(x)$ and
\begin{align}\label{JBrd}
		J^c_{s}(x)= \left(\nabla c_{s,1}(x), \ldots, 	\nabla c_{s,m}(x)\right)^T,
\end{align}
as the approximate Jacobian matrix of $c(x)$, where $\nabla c_{s,i}(x)$ is computed as \eqref{nabrrd}.

\subsection{A derivative-free regularization method for solving \eqref{prob}}
The method to be proposed is based on the framework of the augmented Lagrangian method.	
The  Lagrangian function for \eqref{prob} is
\begin{align}\label{Lag}
L(x,y)=\frac{1}{2}\|r(x)\|^2 -y^Tc(x),
\end{align}
where $y\in \mathbb{R}^{m}$ is the Lagrange multiplier.
At the $k$-th iteration, fixing the Lagrange multiplier $y_k$ and the penalty parameter  $\delta_k^{-1}$, the augmented Lagrangian method solves the subproblem
\begin{align}%\label{r1}
		\min_{x\in\mathbb{R}^{n}} \ \frac{1}{2}\|r(x)\|^2-y_k^Tc(x)+ \frac{1}{2}\delta_k^{-1} \|c(x)\|^2,
	\end{align}
which is equivalent to
	\begin{align}\label{r1}
		\min_{(x,v)\in\mathbb{R}^{n+m}} \quad
		&\frac{1}{2}\|r(x)\|^2  	+ \frac{1}{2}\delta_k 	\|v+y_k\|^2\nonumber\\
		\text{s.t.}\quad\ \quad & c(x)+\delta_k v=0.
	\end{align}
In \cite{Orban2020}, Orban et al. apply the regularization scheme and introduce $z=r(x)$
to avoid the occurrence of $J^r(x)^TJ^r(x)$, where $J^r(x)$ is the Jacobian matrix of $r(x)$,
which typically contributes to dense and ill conditioned subproblems.
In this paper, we follow their idea and solve the subproblem:
	\begin{align}\label{rglar}
		\min_{(x,v,z)\in\mathbb{R}^{n+m+p}} \quad
		&\frac{1}{2}\|z\|^2+ \frac{1}{2} \rho_k\|x-x_k\|^2+	\frac{1}{2}\delta_k	\|v+y_k\|^2\nonumber\\
		\text{s.t.} \qquad \quad & r(x)-z=0,\nonumber\\
		&c(x)+\delta_k v=0,
	\end{align}
where $\rho_k$ is a regularization parameter.
%{\color{red} Note that (\ref{rglar})  satisfies the linear independence constraint qualification when $\delta_k>0$.}
%Second, when applying Newton-like   method to (\ref{rglar}), it avoids the calculation of Jacobian matrix
%	multiplication in $J^TJ$ form.
	
Note that the Jacobian matrices of $r(x)$ and $c(x)$ considered in this paper
are unavailable or expensive to compute.
We define the approximate KKT conditions for (\ref{rglar}) by replacing the exact Jacobian matrices with approximate ones:
\begin{align}\label{KKT1}
		\rho_k(x-x_k)-J_s^r(x)^Ts-J_s^c(x)^Ty&=0, \nonumber\\
		z+s&=0, \nonumber\\
		\delta_k(v+y_k)-\delta_k y&=0,\\
		r(x)-z&=0, \nonumber\\
		c(x)+\delta_k v&=0, \nonumber
\end{align}
where $J^r_{s}(x)$ and $J^c_{s}(x) $ are approximate Jacobian matrices of $r(x)$ and $c(x)$ respectively, $s$ and $y$ are Lagrange multipliers.
Since $s=-z$ and $v=y-y_k$,   (\ref{KKT1}) becomes
	\begin{align}\label{KKT2}
		\rho_k(x-x_k)+J_s^r(x)^Tz-J_s^c(x)^Ty&=0, \nonumber\\
		r(x)-z&=0, \\
		c(x)+\delta_k (y-y_k)&=0. \nonumber
	\end{align}

	Denote $w:=(x,z,y)$, $w_k:=(x_k,z_k,y_k)$, and define
	\begin{align}\label{Grd}
		F_{s_k}(w;x_k,y_k,\rho_k,\delta_k):=	\begin{bmatrix}
			\rho_k(x-x_k)+J^r_{s_k}(x)^Tz-J^c_{s_k}(x)^Ty\\
			r(x)-z\\
			c(x)+\delta_k (y-y_k)
		\end{bmatrix},
	\end{align}
	where $J^r_{s_k}(x)$ and $J^c_{s_k}(x)$ are defined by (\ref{JArd}) and (\ref{JBrd}) with $\gamma_k$.
One can apply Newton-like method to
\begin{align}\label{newton}
	F_{s_k}(w;x_k,y_k,\rho_k,\delta_k)=0
\end{align}
sequentially to find an approximate KKT point.
Let $T_{s_k}$ be an approximation of the Jacobian matrix of $F_{s_k}(w;x_k,y_k,\rho_k,\delta_k)$ at $w_k$.
 Denote $J^r_{s_k}:=J^r_{s_k}(x_k)$, $J^c_{s_k}:=J^c_{s_k}(x_k)$  and $F_{s_k}:=F_{s_k}(w_k;x_k,y_k,\rho_k,\delta_k)$.
	%$r_k:=r(x_k)$, %$c_k:=c(x_k)$
%	\begin{align}\label{Gkrd}
%		F_{s_k}(w_k):=F_{s_k}(w_k;x_k,y_k,\rho_k,\delta_k)=
%		\begin{bmatrix}
%			(J^r_{s_k})^Tz_k-(J^c_{s_k})^Ty_k\\
%			r(x_k)-z_k\\
%			c(x_k)
%		\end{bmatrix}.
%	\end{align}
Then solving the linear system
\begin{align}
T_{s_k} d_w=-F_{s_k}
\end{align}
can give an inexact solution of \eqref{newton}.
In practice, one can adjust $d_y$ to $-d_y$ to make the coefficient matrix of the linear system symmetric so that a symmetric indefinite factorization can be used (cf. \cite{MA57}).
That is, an inexact Newton step $d_{w_k}=(d_{x_k}, d_{	z_k}, d_{y_k} )$ is computed by solving the system
	%$T_k\Delta w_k=-F_{s_k}$
	% which can be written as
	\begin{align} \label{sysrd}
		\begin{bmatrix}
			H_{s_k}+\rho_k I & (J^r_{s_k})^T & (J^c_{s_k})^T \\
			J^r_{s_k} & -I & 0\\
			J^c_{s_k} & 0 & -\delta_k I
		\end{bmatrix}\begin{bmatrix}
			d_{x}  \\
			d_{	z}\\
			-d_{y}
		\end{bmatrix}
		=\begin{bmatrix}
			-(J^r_{s_k})^Tz_k+(J^c_{s_k})^Ty_k \\
			z_k-r(x_k) \\
			-c(x_k)
		\end{bmatrix},
	\end{align}
where $H_{s_k}$ is a symmetric approximation of $H(x,z_k,y_k)$ at $x_k$ with
	\begin{align}\label{Hrd}
		H(x,z,y)=\sum_{i=1}^{p} z_i \nabla^2 r_i(x)
		-\sum_{i=1}^{m} y_i \nabla^2
		c_i(x).
	\end{align}
	
Define the merit function
\begin{align*}
		\|F_{s_k}\|_*=\|(J^r_{s_k})^T z_k-(J^c_{s_k})^T y_k\|+\|z_k-r(x_k)\|+	\|c(x_k)\|.
\end{align*}
If the trial step $d_{w_k}$	can lead to a sufficient decrease, i.e.,
\begin{align}\label{innerstoprd}
		\|F_{s_{k+1}}\|_* \leq \theta\|F_{s_k}\|_*+\epsilon_k,
\end{align}
where $0<\theta<1$, $0< \epsilon_k\downarrow 0$,
we accept it and  take $w_{k+1}=w_k+d_{w_k}$.
Otherwise we go to inner iterations to minimize the augmented Lagrangian function so that a trial step can be obtained to satisfy \eqref{innerstoprd}. That is, we solve
\begin{align}\label{alrd}
	\min_x\	\phi_k(x;\delta)=\frac{1}{2}\|r(x)\|^2-y_k^Tc(x)+
		\frac{1}{2}\delta^{-1}\|c(x)\|^2,
\end{align}
where $\delta$ may be updated in inner iterations.

Actually, \eqref{alrd} is equivalent to the nonlinear least square problem
	\begin{align}\label{inLS}
		\min_x\ \frac{1}{2}\|\Phi_k(x;\delta)\|^2:=\frac{1}{2}
		\left\|\begin{bmatrix}
			r(x)\\
			\delta^{-1/2}(c(x)-\delta y_k)
		\end{bmatrix}\right\|^2.
	\end{align}
It can be solved by a derivative-free LM method.
Let
\begin{align*}
		J^{\Phi}_s(x;\delta)=\begin{bmatrix}
			J^r_s(x)\\
			\delta^{-1/2}J^c_s(x)
		\end{bmatrix}
	\end{align*}
and
\begin{align}\label{phik}
\nabla\phi_{k,s}(x;\delta)=J^r_{s}(x)^Tr(x)-J^c_{s}(x)^Ty_k +\delta^{-1} J^c_{s}(x)^T c(x)
\end{align}
be the approximate Jacobian matrix of $\Phi_k(x;\delta)$ and approximate gradient of $\phi_k(x;\delta)$ with respect to $x$, respectively.
At the $j$-th inner iteration of the $k$-th outer iteration, we denote
$$
J_{s_{k,j}} := J^{\Phi}_s(x_{k,j};\delta_{k,j}),
\quad \nabla \phi_{s_{k,j}}:= \nabla \phi_{k,s}(x_{k,j};\delta_{k,j})
$$
for brevity, and solve
\begin{align}\label{LMstep}
(J_{s_{k,j}}^TJ_{s_{k,j}}+\lambda_{k,j} \|\nabla \phi_{s_{k,j}}\| I)d_x=- \nabla \phi_{s_{k,j}}
\end{align}
to obtain the step $d_{ x_{k,j}}$, where $\lambda_{k,j}$ is the parameter updated from iteration to iteration.
We may use two-step QR factorization such as that of \cite{Sun2006} to find the solution of \eqref{LMstep},
which can avoid the calculation of $J_{s_{k,j}}^TJ_{s_{k,j}}$.
Define the ratio
\begin{align}\label{LMratio}
\zeta_{k,j}=\frac{\|\Phi_k(x_{k,j};\delta_{k,j})\|^2-\|\Phi_k(x_{k,j}+d_{ x_{k,j}};\delta_{k,j})\|^2}
				{\|\Phi_k(x_{k,j};\delta_{k,j})\|^2-\|\Phi_k(x_{k,j};
					\delta_{k,j})+J_{s_{k,j}}d_{ x_{k,j}}\|^2}.
\end{align}
It plays an important role in deciding whether $d_{ x_{k,j}}$ is acceptable and how to update $\lambda_{k,j}$.
The parameter $\delta_{k,j}$ will be decreased if
the relatively sufficient decrease on $\|\nabla \phi_{k,s}(x ;\delta)\|$ is achieved
but the constraint violation decrease does not meet the requirement.
	
The derivative-free regularization algorithm for solving \eqref{prob} is presented as follows.

		\begin{algorithm}%\label{ALG1}
		\renewcommand{\algorithmicrequire}{\textbf{Input:}}
		\renewcommand{\algorithmicensure}{\textbf{Output:}}
		\caption{A derivative-free regularization algorithm for solving \eqref{prob}}
		\label{alg1}
		\begin{algorithmic}[1] %用数字每行标号
			\REQUIRE $x_0, \delta_0, \rho_0, \epsilon_0, \gamma_0, \theta\in(0,1)$, set $k:=0$.
			\STATE  Initialize approximate Jacobians $J^r_{s_0}$ and $J^c_{s_0}$ with $\gamma_0$.
			%				by (\ref{JA}) and (\ref{nabr}).
			\STATE If a stopping condition is  satisfied, stop.
%			\STATE Update $H_{s_k}$ and $\epsilon_k$.  % with $A_{m_k}$
%			and
			%$B_{m_k}$.
			\STATE  Update $H_{s_k}$, and solve (\ref{sysrd}) to
			obtain $d_{w_k}$.
			\STATE Set $\hat{w}_k=w_k+d_{w_k}$.
			\STATE Compute $J^r_{\hat{s}_k}$ and $J^c_{\hat{s}_k}$ with $\hat{\gamma}_k=\|\hat{x}_k-x_k\|$.
			\STATE If $\|F_{\hat{s}_k}\|_*\leq \theta\|F_{s_k}\|_*+\epsilon_k$,  set
			 $$w_{k+1}=\hat{w}_k, 	\gamma_{k+1}=\hat{\gamma}_k,
			J^r_{s_{k+1}}=J^r_{\hat{s}_k}, J^c_{s_{k+1}}=J^c_{\hat{s}_k},$$
             and  update $\delta_{k+1}>0, \rho_{k+1} \ge 0$.
			
Otherwise, run Algorithm \ref{alg2} to obtain $w_{k+1}$, $J^r_{s_{k+1}}$, $J^c_{s_{k+1}}$,  $\delta_{k+1}$ and $\rho_{k+1}$.

			\STATE Update $\epsilon_{k+1}$.
			\STATE
			Set $k:=k+1$. Go to line 2.
			
			%		\ENDFOR
			\ENSURE $x_k$, $f(x_k)$, $\|c(x_k)\|$.
		\end{algorithmic}
	\end{algorithm}

	\begin{algorithm}%\label{ALG1}
		\renewcommand{\algorithmicrequire}{\textbf{Input:}}
		\renewcommand{\algorithmicensure}{\textbf{Output:}}
		\caption{A derivative-free LM algorithm for \eqref{inLS}}
		\label{alg2}
		\begin{algorithmic}[1] %用数字每行标号
			\REQUIRE $w_{k,0}=w_k$, %$y_{k,0}=y_k$,
			$F_{s_k}$,
			$\epsilon_k$, $J^r_{s_{k,0}}=J^r_{s_k}$,
			$J^c_{s_{k,0}}=J^c_{s_k}$,	$0<p_0<p_1<1<p_2$,
			$0<p_3<1<p_4$,
			$0<\lambda_{\min}<\lambda_{k,0}$,
			$\gamma_{k,0}=\gamma_{k}$,
			$\delta_{k,0}=\delta_{k,-1}=\delta_{k}$, set $j:=0$.
%			\WHILE{$\|F_{s_{k,j}}(w_{k,j})\|_*
%				> \theta\|F_{s_k}\|_*+\epsilon_k$}
			\STATE
			Solve (\ref{LMstep}) to obtain $d_{x_{k,j}} $.
			\STATE
			Compute  $\zeta_{k,j}$ by \eqref{LMratio}.
%			where
%			\begin{align}\label{Phi}
%				\Phi_k(x;\delta)=\begin{bmatrix}
%					r(x)\\
%					\delta^{-1/2}(c(x)-\delta y_k)
%				\end{bmatrix},
%				J_{s}=\begin{bmatrix}
%					A_s\\
%					\delta^{-1/2}B_s
%				\end{bmatrix}.
%			\end{align}
			\STATE Set
			\begin{align}
				x_{k,j+1}=
				\begin{cases}
					x_{k,j}+d_{ x_{k,j}},& \mbox{if}\ \zeta_{k,j}\ge p_0,\\
					x_{k,j}, & \mbox{otherwise}.
				\end{cases}
			\end{align}
		\STATE	If $\zeta_{k,j}< p_0$, compute $\lambda_{k,j+1}=4\lambda_{k,j}$.
			Otherwise, compute
			\begin{equation}\label{2.7a}
				\lambda_{k,j+1}=
				\begin{cases}
					4\lambda_{k,j},  & \mbox{if}\  \|J_{s_{k,j}}^T
					\Phi_k(x_{k,j};\delta_{k,j})\|<
					\frac{p_1}{\lambda_{k,j}},\\
					\lambda_{k,j},   & \mbox{if}\
					\|J_{s_{k,j}}^T \Phi_k(x_{k,j};\delta_{k,j})\|\in
					[\frac{p_1}{\lambda_{k,j}},\frac{p_2}{\lambda_{k,j}}),\\
					\max\{0.25\lambda_{k,j},\lambda_{\min}\}, &
					\mbox{otherwise}.
				\end{cases}
			\end{equation}
			\STATE Set $z_{k,j+1}=r(x_{k,j+1}).$
			\STATE If $\delta_{k,j}=\delta_{k,j-1}$, compute $\gamma_{k,j+1}=\min\{\frac{1}{2}\gamma_{k,j}, \|d_{x_{k,j}}\| \}$.
			%		{\color{red} or
				%			$\gamma_{k,j+1}=(1-a_A)
				%			
				%\|\nabla\phi_{k,\gamma_{k,j}}(x_{k,j};\delta_{k,j})\|
				%			/ \|\Delta x_{k,j}\|$. NO!}
Otherwise, compute
			\begin{equation}\label{gamaupd}
				\gamma_{k,j+1}=
				\begin{cases}
					\frac{1}{2} \gamma_{k,j}, &	\mbox{if}\	\|(J^c_{s_{k,j}})^T c(x_{k,j})\|<
					p_3\gamma_{k,j},\\
               		\gamma_{k,j},   & \mbox{if}\
					\|(J^c_{s_{k,j}})^T
					c(x_{k,j})\|\in
					[p_3\gamma_{k,j},p_4\gamma_{k,j}),\\
  \min\{2\gamma_{k,j},\gamma_{k}\},    &
					\mbox{otherwise}.
				\end{cases}
			\end{equation}

			%			\IF{$\rho^{LM}_j\ge p_0$}
			%			\STATE Set $	\gamma_{k,j+1}= \frac{1}{2}	
			%\gamma_{k,j}$.
			%			\ELSE
			%			\STATE Set $	\gamma_{k,j+1}= \|d_j\| $.
			%			\ENDIF
			%			\STATE  Set $	
			%\gamma_{k,j+1}=\min\{\|d_j\|,\frac{1}{2}	
			%			\gamma_{k,j}\}  $.
			\STATE
			Compute $J^r_{s_{k,j+1}}$ and $J^c_{s_{k,j+1}}$ with $\gamma_{k,j+1}$.
			%		and 		$H_{s_{k,j+1}}$.
		\STATE	If $\|\nabla_xL_{s_{k,j+1}}(x_{k,j+1},y_k-\delta_{k,j}^{-1}
				c(x_{k,j+1}))  \| \leq \theta \|\nabla_xL_{s_{k}} (x_k,y_k)
				\|+	\frac{1}{2}\epsilon_k$ and $\|c(x_{k,j+1})\|>
				\theta \|c(x_{k})\|		+\frac{1}{2}\epsilon_k $, compute
			 $\delta_{k,j+1}=\frac{1}{10}\delta_{k,j}$.			
 Otherwise, set $\delta_{k,j+1}= \delta_{k,j}$.
			%\STATE $\omega^{k+1}=\omega^k+d^k$.
			\STATE Set $j:=j+1$.
		Compute	$y_{k,j}=y_k-\delta_{k,j}^{-1}c(x_{k,j})$.
%			\ENDWHILE

\STATE If $\|F_{s_{k,j}}(w_{k,j})\|_* >  \theta\|F_{s_k}\|_*+\epsilon_k$, go to line 1.
Otherwise, set $w_{k+1}=w_{k,j}$, $J^r_{s_{k+1}}=J^r_{s_{k,j}}$, $J^c_{s_{k+1}}=J^c_{s_{k,j}}$,
			$\delta_{k+1}=\delta_{k,j}$, $\rho_{k+1}=\lambda_{k,j} \|\nabla
			\phi_{k,s_j}(x_{k,j};\delta_{k,j})\| $, and return to Algorithm \ref{alg1}.
			
		\end{algorithmic}
	\end{algorithm}

 %From this point of view,  our algorithm can be regarded as an  hybrid method because of the  Newton-like step used in main iteration and the LM step used in Algorithm \ref{alg2}.

%\begin{align}\label{qr1}
%	Q J_s P=R,
%\end{align}	
%where $P$ is the permutation matrix, $Q$ is an orthogonal matrix and $R$ is an
%upper triangular matrix.
%(\ref{qr1}) is equivalent to
%\begin{align}\label{qr3}
% \begin{bmatrix}
% 	Q &0	\\
% 	0 & P^T
% \end{bmatrix}\begin{bmatrix}
% J_s	\\
% \rho^{-1/2} I
%\end{bmatrix}P=\begin{bmatrix}
%R	\\
%\rho^{-1/2} I
%\end{bmatrix}.
%\end{align}
% (\ref{LMstep}) can be viewed as the normal equation of least Squares problem
% \begin{align}\label{nls}
% 	\begin{bmatrix}
% 	J_s	\\
% 		\rho^{-1/2} I
% 	\end{bmatrix}d_x \cong -\begin{bmatrix}
% 	\Phi_k	\\
% 	0
% \end{bmatrix}.
% \end{align}
% By (\ref{nls}) and (\ref{qr3}), there is
% \begin{align}\label{qr4}
% 	\begin{bmatrix}
% 		R	\\
% 		\rho^{-1/2} I
% 	\end{bmatrix} P^T d_x= - 	\begin{bmatrix}
% 	Q\Phi_k	\\
% 	0
% \end{bmatrix}.
% \end{align}
%Then do QR decomposition for $[R; \rho^{-1/2} I]$,
%\begin{align}\label{qr2}
%	W	\begin{bmatrix}
%		R	\\
%		\rho^{-1/2} I
%	\end{bmatrix} =
%	\zeta_{\rho}.	
%\end{align}
%From (\ref{qr4}) and (\ref{qr2}),
%\begin{align}\label{qr5}
%	\zeta_{\rho}P^Td_x =-W\begin{bmatrix}
%		Q\Phi_k	\\
%		0
%	\end{bmatrix}:=-\begin{bmatrix}
%	v_1	\\
%	v_2
%\end{bmatrix}.
%\end{align}
%Then
%\begin{align}\label{qr}
%	d_x=-P \zeta_{\rho}^{-1}v_1.
%\end{align}

\section{Global convergence of the algorithm}\label{sec3}
In this section, we study the global convergence of Algorithm \ref{alg1}.	
It is shown that Algorithm \ref{alg1} either finds an approximate KKT point with high probability
or converges to a stationary point of the constraints almost surely.

Note that the Jacobian approximations are built in some random fashion,
which implies the randomness of the iterates, the smoothing parameters, etc.
We use uppercase letters to represent random quantities, such as $X_k$, $J^r_{S_k}(X_k)$,  etc.,
and use lowercase letters to represent their realizations, such as $x_k$,
$J^r_{s_k}(x_k)$ etc.
Denote by $\mathcal{F}_{k}$  the $\sigma$-algebra generated by all randomness in
iteration $k$ and all previous iterations.

We first derive the bounds for the difference between the Jacobian matrix and the expectation of its  approximation and the variance of the approximate Jacobian, respectively.
We make the following assumption.
	
\begin{assumption}\label{ass1rd}
(i) $J^r(x)$ and $J^c(x)$ are Lipchitz continuous, i.e., there exist $\kappa_{ljr}, \kappa_{ljc}>0 $  such that for all $x,y \in\mathbb{R}^n$,
\begin{align}
			\|J^r(x)-J^r(y)\| &\leq \kappa_{ljr} \|x-y\|,  \label{lipaed}\\
			\|J^c(x)-J^c(y)\| &\leq \kappa_{ljc} \|x-y\|. \label{lipbed}
\end{align}

(ii) $r(x)$ and $c(x)$ are Lipschitz continuous, i.e., there exist $\kappa_{lr}, \kappa_{lc}>0$ such that for all $x,y \in		\mathbb{R}^n$,
\begin{align}\label{liprrd}
			\|r(x)-r(y)\|& \leq \kappa_{lr} \|x-y\|, \\
			\|c(x)-c(y)\|& \leq \kappa_{lc} \|x-y\|.
\end{align}
		
 \end{assumption}	
	
	\eqref{lipaed} implies that there exists $\kappa_{lgr}>0 $ such that for all $x,y \in\mathbb{R}^n$, 		\begin{align}\label{lrird}
			\|\nabla r_i(x)- \nabla r_i(y)\| \leq \kappa_{lgr} \|x-y\|, \quad i=1,\ldots,p.
		\end{align}
\eqref{lipbed} implies that there exists 	$\kappa_{lgc}>0 $ such that for all $x,y \in\mathbb{R}^n$,
		\begin{align}\label{lcird}
			\|\nabla c_i(x)- \nabla c_i(y)\| \leq \kappa_{lgc}\|x-y\|,\quad i=1,\ldots,m.
		\end{align}

	%\begin{assumption}\label{ass5}			
	%			\begin{align}\label{phinf}
		%				\inf_j	\phi_k(x_{k,j};\delta_{k,j})>-\infty,
		%				\end{align}
	%\end{assumption}	
	
	%{\color{red} indexes in  alg2orithm are not changed yet}

	%The following Lemma provide  error bound for the
	%derivative-free estimators (\ref{JArd}) and (\ref{JBrd}).

\begin{lemma}\label{lemAB}
Suppose that Assumption \ref{ass1rd}(i) holds. Then there exists $d>0$ such that
\begin{align}
			\|\mathbb{E}_u [J^r_{S}(x)]-J^r(x)\| &\leq\frac{d\sqrt{p}n	\kappa_{lgr}}{n+1}	\gamma,		\label{lem1A}\\
			\|\mathbb{E}_u [J^c_{S}(x)]-J^c(x)\| &\leq	\frac{d\sqrt{m}n\kappa_{lgc}}{n+1}\gamma,		\label{lem1B}
\end{align}
and
\begin{align}
			&\mathbb{E}_u \left[ \left\|J^r_{S}(x)-\mathbb{E}_u
			[J^r_{S}(x)] 	\right\|^2\right] \leq  \frac{d^2p n\kappa_{lgr}^2}{4}\gamma^2, \label{lem1varA}\\
			& \mathbb{E}_u \left[ \left\|J^c_{S}(x)-\mathbb{E}_u
			[J^c_{S}(x)] 	\right\|^2\right] \leq  \frac{d^2m n\kappa_{lgc}^2}{4}\gamma^2. \label{lem1varB}
\end{align}
\end{lemma}
	
\begin{proof}
		By the equivalence of matrix norms, there exists  $d>0$ such that
		\begin{align}
			&\|\mathbb{E}_u [J^r_{S}(x)]-J^r(x)\| \leq d \| \mathbb{E}_u[J^r_{S}(x)]-J^r(x)\|_F, \label{normeq} \\
	%		&\|\mathbb{E}_u [J^c_{S}(x)]-J^c(x)\| \leq d \|	\mathbb{E}_u[J^c_{S}(x)]-J^c(x)\|_F, \label{normeq1}
		&	\left\|J^r_{S}(x)-\mathbb{E}_u 	[J^r_{S}(x)] \right\| \leq d \left\|J^r_{S}(x)-\mathbb{E}_u [J^r_{S}(x)] 	\right\|_F. \label{normeq2}
		%&	\left\|J^c_{S}(x)-\mathbb{E}_u 	[J^c_{S}(x)] 	\right\| \leq d \left\|J^c_{S}(x)-\mathbb{E}_u [J^c_{S}(x)] 	\right\|_F.
		\end{align}
		It follows from \cite{wang} that
		\begin{equation*}
			\|\mathbb{E}_u [\nabla r_{S,i}(x)] - \nabla r_i(x)\| \leq
			\frac{n	\kappa_{lgr}}{n+1}\gamma.
			%\|\mathbb{E}_u [\nabla c_{S,i}(x)] - \nabla c_i(x)\| \leq
%			\frac{n	\kappa_{lgc}}{n+1} \gamma.
		\end{equation*}
Hence,
\begin{align*}
\| \mathbb{E}_u[J^r_{S}(x)]-J^r(x)\|_F^2&=\sum_{i=1}^{p}\|\mathbb{E}_u
[\nabla r_{S,i}(x)] - \nabla r_i(x)\|^2 \leq p\left(\frac{n	\kappa_{lgr}}{n+1}\right)^2 \gamma^2.
%\| \mathbb{E}_u[J^c_{S}(x)]-J^c(x)\|_F^2
%&=\sum_{i=1}^{m}\|\mathbb{E}_u	[\nabla c_{S,i}(x)] - \nabla c_i(x)\|^2 \leq m\left(\frac{n	\kappa_{lgc}}{n+1}\right)^2\gamma^2.
		\end{align*}
This, together with (\ref{normeq}), yields (\ref{lem1A}).

		On the other hand, it holds that for $i=1,\ldots,p$,
		\begin{align}\label{Ers}
			\mathbb{E}_u\left[ \left\|\nabla r_{S,i}(x)-\mathbb{E}_u
			[\nabla r_{S,i}(x)] 	\right\|^2\right] =\mathbb{E}_u \left[ \|\nabla r_{S,i}(x)
			\|^2\right] -\|\mathbb{E}_u [\nabla r_{S,i}(x)] \|^2.
		\end{align}
		By the Taylor expansion,
		\begin{align*}
			\frac{r_i(x+\gamma u_j)-r_i(x)}{\gamma} = \nabla
			r_i(x)^T
			u_j+ \frac{1}{2} \gamma u_j^T r''_i(\xi_i)u_j,
		\end{align*}
		where $\xi_i\in [x,x+\gamma u_j]$.
		Denote $\beta_j=u_j^T r''_i(\xi_i) u_j, j=1,\ldots, n$.
By (\ref{lrird}),
		\begin{align}\label{K}
			|\beta_j|\leq \|u_j\|^2\|r''_i(\xi_i) \| \leq \kappa_{lgr}.
		\end{align}
 Since $u_j^Tu_j=1$ and $u_j^Tu_k=0$ for all $j\not=k, j,k=1,\ldots,n$, and
		\begin{equation}\label{uu}
			\mathbb{E}_u [u_j u_j^T]=\frac{1}{n}I,
		\end{equation}
(cf. \cite{wang}),   by (\ref{nabrrd}) and (\ref{K}),
		\begin{align}\label{tem1}
			&\mathbb{E}_u \left[ \| \nabla r_{S,i}(x)	\|^2\right]
			=\mathbb{E}_u \left[\left\| \sum_{j=1}^n \frac{r_i(x+\gamma u_j)-r_i(x)}{\gamma} u_j \right\|^2\right] \nonumber\\
			&= \mathbb{E}_u \left[\left\| \sum_{j=1}^n u_j^T \nabla r_i(x)	u_j+ \frac{\gamma}{2}  \sum_{j=1}^n   \beta_j  u_j
			\right\|^2\right] \nonumber\\
& =\sum_{j=1}^n\mathbb{E}_u \left[ \nabla
			r_i(x)^T u_j u_j^T \nabla r_i(x) \right] + \gamma
			\sum_{j=1}^n  \mathbb{E}_u \left[ \nabla r_i(x)^T u_j \beta_j
			\right]  +
			\frac{\gamma^2}{4}  \sum_{j=1}^n \mathbb{E}_u \left[ \beta_j^2
			\right]
			\nonumber\\
			& \leq \|\nabla	r_i(x)\|^2 +\gamma
			\sum_{j=1}^n  \mathbb{E}_u \left[ \nabla r_i(x)^T u_j \beta_j
			\right]  +\frac{ n \kappa_{lgr}^2}{4}\gamma^2.
		\end{align}

		Now we consider $\|\mathbb{E}_u [\nabla r_{S,i}(x) ] \|^2$. By	\eqref{uu},
		\begin{align}\label{tem2}
			&	\|\mathbb{E}_u [\nabla r_{S,i}(x) ] \|^2 = \left\|
			\mathbb{E}_u\left[\sum_{j=1}^n\frac{r_i(x+\gamma u_j)-r_i(x)}{\gamma} u_j   \right]	\right\|^2  \nonumber \\
			&=\left\| \mathbb{E}_u \left[ \sum_{j=1}^n	\left(u_j^T	\nabla r_i(x)\right) u_j\right] + \frac{\gamma}{2}
			\sum_{j=1}^n\mathbb{E}_u\left[ \beta_j u_j \right] \right\|^2\nonumber\\
			&=\left\|\sum_{j=1}^n \mathbb{E}_u \left[	u_ju_j^T\nabla r_i(x) \right] \right\|^2 + \gamma \left(
			\sum_{j=1}^n\mathbb{E}_u \left[  u_ju_j^T \nabla r_i(x) \right] \right)^T\left(
			\sum_{j=1}^n\mathbb{E}_u 	\left[ \beta_j u_j \right] \right)\nonumber\\
			& \quad + \frac{ \gamma^2}{4}  \left(\sum_{j=1}^n\mathbb{E}_u 	\left[ \beta_j u_j \right]	\right)^T\left(
			\sum_{j=1}^n\mathbb{E}_u 	\left[ \beta_j u_j \right]	\right)\nonumber\\
			& =  \|\nabla	r_i(x)\|^2 +\gamma\sum_{j=1}^n  \mathbb{E}_u \left[ \nabla r_i(x)^T u_j \beta_j\right] +
			\frac{\gamma^2}{4}  \left(\sum_{j=1}^n\mathbb{E}_u 	\left[ \beta_j u_j \right]\right)^T\left(			\sum_{j=1}^n\mathbb{E}_u 	\left[ \beta_j u_j \right] \right).
		\end{align}
		Combining \eqref{Ers},   (\ref{tem1}) and (\ref{tem2}), we obtain
		\begin{align}\label{varr}
			 \mathbb{E}_u\left[ \left\|\nabla r_{S,i}(x)-\mathbb{E}_u
			[\nabla r_{S,i}(x)] 	
			\right\|^2\right]
			& \leq  \frac{ n \kappa_{lgr}^2}{4}\gamma^2.
		\end{align}
		Thus by (\ref{normeq2}) and (\ref{varr}),
		\begin{align*}
			&\mathbb{E}_u\left[ \left\|J^r_{S}(x)-\mathbb{E}_u
			[J^r_{S}(x)] 	
			\right\|^2\right]
			\leq d^2 \mathbb{E}_u\left[ \left\|J^r_{S}(x)-\mathbb{E}_u
			[J^r_{S}(x)] \right\|_F^2\right] \nonumber\\
			& =d^2 \sum_{i=1}^{p} \mathbb{E}_u\left[ \left\|\nabla
			r_{S,i}(x)-\mathbb{E}_u
			[\nabla r_{S,i}(x)] 	
			\right\|^2\right] \leq \frac{ d^2 pn \kappa_{lgr}^2}{4}\gamma^2,
		\end{align*}
		which gives (\ref{lem1varA}).
	
	Similarly, \eqref{lem1B} and (\ref{lem1varB}) can be proved.
	\end{proof}

Next we give the bound for the difference between the gradient of the augmented Lagrangian function
and the expectation of its  approximation in inner iterations,
and the bound for the variance of the approximate gradient as well.

\begin{assumption}\label{ass2rd}
(i) The approximate Jacobian matrices  $J^r_{s_k}, J^c_{s_k}, J^r_{s_{k,j}}, J^c_{s_{k,j}}$ are bounded for all $k$ and $j$.
	
(ii) The function values  $r_k, c_k, r_{k,j}, c_{k,j}$ are bounded for all $k$ and $j$, i.e., there exist 	$\kappa_{br}, \kappa_{bc}>0$ such that
\begin{align}\label{rckbounded}
	\|r(x_k) \|\leq \kappa_{br}, \ \ \|c(x_k)\| \leq \kappa_{bc},\ \
\|r(x_{k,j})  \|\leq \kappa_{br}, \ \  \|c(x_{k,j})\| \leq \kappa_{bc}.
\end{align}

(iii) The multipliers $z_k, y_k$ are bounded, i.e., there exist
$\kappa_{bz}, \kappa_{by}>0$		such that
$$\|z_k\| \leq \kappa_{bz}, \ \
\|y_k\| \leq \kappa_{by}.$$

(iv) $T_{s_k}$ are non-singular for all $k$.

(v) $\epsilon_k\downarrow 0$.

\end{assumption}

Note that the non-singularity of $T_{s_k}$  can be guaranteed by setting appropriate $\rho_k$.	

\begin{lemma}\label{lemphi}
Suppose that Assumptions \ref{ass1rd}(i) and \ref{ass2rd}(ii),(iii)  hold.
Then we have
\begin{align}
&	\| \mathbb{E}_{u}[\nabla\phi_{S_{k,j}}(x_{k,j};\delta )]-\nabla	\phi_{k}(x_{k,j};\delta)\| \nonumber\\
& \leq \left(\frac{d\sqrt{p}n\kappa_{lgr}\kappa_{br}}{n+1}+\frac{d\sqrt{m}n\kappa_{lgc}}{n+1}(\|y_k\|+\delta^{-1} \kappa_{bc})	\right) \gamma_{k,j} := \kappa_{\phi_k,1}\gamma_{k,j},
			\label{phi1bound} \\
&\mathbb{E}_u\left[ \left\|\nabla\phi_{S_{k,j}}(x_{k,j};\delta)-\mathbb{E}_u
			[\nabla	\phi_{S_{k,j}}(x_{k,j};\delta )] 	\right\|^2\right] \nonumber\\
&\leq3\left( \frac{d^2p n\kappa_{lgr}^2\kappa_{br}^2}{4}+
			\frac{d^2m n\kappa_{lgc}^2}{4}(\|y_k\|^2+\delta^{-1}\kappa_{bc}^2)\right) \gamma_{k,j}^2 := \kappa_{\phi_k,2}\gamma_{k,j}^2.
			\label{phi2bound}
\end{align}
	
\end{lemma}

\begin{proof}
By the definition of $\phi_k(x;\delta)$,
		\begin{align*}
			\nabla\phi_{S_{k,j}}(x_{k,j};\delta)&
			=(J^r_{S_{k,j}})^Tr(x_{k,j})-(J^c_{S_{k,j}})^Ty_k+\delta^{-1}
			(J^c_{S_{k,j}})^Tc(x_{k,j}),\\
			\nabla\phi_{k}(x_{k,j};\delta)&=
			(J^r_{k,j})^Tr(x_{k,j})-(J^c_{k,j})^Ty_k+\delta^{-1}(J^c_{k,j})^T
			c(x_{k,j}).
		\end{align*}
By Assumption \ref{ass2rd}(ii)--(iii)  and Lemma \ref{lemAB},
		\begin{align*}
			&	\| \mathbb{E}_u[\nabla	\phi_{k,S_{k,j}}(x_{k,j};\delta)]-\nabla	\phi_{k}(x_{k,j};\delta)\| \\
			& = \left \|\left(  \mathbb{E}_u [J^r_{S_{k,j}}]-J^r_{k,j}\right)
			^Tr(x_{k,j})-\left(  \mathbb{E}_u [J^c_{S_{k,j}}]-J^c_{k,j}\right)^Ty_k+\delta^{-1}
			\left(  \mathbb{E}_u [J^c_{S_{k,j}}]-J^c_{k,j}\right)^Tc(x_{k,j})\right\|\\
			& \leq \left\| \mathbb{E}_u [J^r_{S_{k,j}}]-J^r_{k,j}\right\|\|r(x_{k,j})\|+
			\left\|\mathbb{E}_u [J^c_{S_{k,j}}]-J^c_{k,j}\right\|\|y_k\|+ \delta^{-1}
			\left\|\mathbb{E}_u [J^c_{S_{k,j}}]-J^c_{k,j}\right\|\|c(x_{k,j})\| \\
			&\leq \left(\frac{d\sqrt{p}n\kappa_{lgr}\kappa_{br}}{n+1}+\frac{d\sqrt{m}n\kappa_{lgc}}{n+1}(\|y_k\|+\delta^{-1} \kappa_{bc})\right)		\gamma_{k,j}.
		\end{align*}
Similarly, we have
		\begin{align*}
			&	\mathbb{E}_u\left[ \left\|\nabla
			\phi_{S_{k,j}}(x_{k,j};\delta)-\mathbb{E}_u
			[\nabla
			\phi_{S_{k,j}}(x_{k,j};\delta)] 	
			\right\|^2\right]\\
			&=\mathbb{E}_u\left[ \left\| 	\left( J^r_{S_{k,j}}-\mathbb{E}_u[
			J^r_{S_{k,j}}] \right) ^Tr(x_{k,j})-\left(
			J^c_{S_{k,j}}-\mathbb{E}_u[J^c_{S_{k,j}}]\right)
			^Ty_k+\delta^{-1}
			\left( J^c_{S_{k,j}}-\mathbb{E}_u[J^c_{S_{k,j}}]\right)^Tc(x_{k,j})
			\right\|^2\right]\\
			&\leq 3\mathbb{E}_u\left[\left\|  J^r_{S_{k,j}}-\mathbb{E}_u[
			J^r_{S_{k,j}}] \right\|^2 \kappa_{br}^2 +\left\|
			J^c_{S_{k,j}}-\mathbb{E}_u[
			J^c_{S_{k,j}}] \right\|^2 \|y_k\|^2+ \delta^{-2}\left\|
			J^c_{S_{k,j}}-\mathbb{E}_u[
			J^c_{S_{k,j}}] \right\|^2 \kappa_{bc}^2  \right]\\
			&\leq 3\left( \frac{d^2p n\kappa_{lgr}^2\kappa_{br}^2}{4}+
			\frac{d^2m n\kappa_{lgc}^2}{4}(\|y_k\|^2+\delta^{-2}\kappa_{bc}^2)\right) \gamma_{k,j}^2.
		\end{align*}	
The proof is completed.	
	\end{proof}

Based on Lemma \ref{lemphi}, we investigate the properties of Algorithm \ref{alg2} if it is called by Algorithm \ref{alg1}.
We discuss two cases: $\Delta_{k,j}$ decrease finitely many times or infinitely many times.
As Algorithm \ref{alg2} is similar to the derivative-free LM algorithm
\cite[Algorithm 1]{cx} for unconstrained nonlinear least squares problems
in the former case, we apply some of its results directly.		
		
\begin{lemma}\label{lembeforthem}
Suppose that Assumptions \ref{ass1rd} and \ref{ass2rd}(i)-(iii) hold for every realization of  Algorithm \ref{alg2} with fix $k$.
If $\Delta_{k,j}$ decrease finitely many times, then
\begin{align}
				\|F_{S_{k,j+1}}\|_* \leq \theta \|F_{s_k}\|_*+\epsilon_k
\end{align}
		holds for some $j$ with any probability $0<p<1$.
\end{lemma}

\begin{proof}
Denote
$$
\mathcal E_1=\{\Delta_{k,j} \ \text{decrease finitely many times}\}.
$$
If $\mathcal E_1$ happens in  Algorithm \ref{alg2} for some $k$,
then there exists $j_0$ such that $\Delta_{k,j}=\Delta_{k,j_0}$ for all $j \ge
j_0$.
So Algorithm \ref{alg2} is similar to \cite[Algorithm 1]{cx}, except the updating rule of the smoothing parameter.
In this paper, we take $\gamma_{k,j+1}=\min\{\frac{1}{2}\gamma_{k,j}, \|d_{x_{k,j}}\|\}$ for all $j\ge j_0$.
Under  Assumptions \ref{ass1rd} and \ref{ass2rd}(i),
according to the proof of  \cite[Theorem 3.2]{cx} and \cite[Remark 3.4]{cx},
$\{\nabla \phi_{S_{k,j}}(X_{k,j};\Delta_{k,j_0})\}_{j\ge j_0}$ is $\frac{1}{2}$-probabilistically
			first-order 	accurate  as defined in \cite[Definition 3.1]{cx}.
Following \cite[Theorem 4.6]{cx}, we have
\begin{equation}\label{limconverg}
\mathbb P\left( 	\liminf_{j \to \infty}\|\nabla\phi_{k}(X_{k,j};\Delta_{k,j_0})\|=0 \ \Big |\ \mathcal E_1 \right)	=1.
\end{equation}  		
Moreover, it follows from \cite[Lemma 4.3]{cx} that $\lim\limits_{j \to \infty}\lambda_{k,j}=+ \infty$. This, together with $\gamma_{k,j} \leq \|d_{x_{k,j-1}}\|\leq \frac{1}{\lambda_{j-1}}$, implies that
\begin{align}\label{gamma}
\lim_{j \to \infty}\gamma_{k,j}=0
\end{align} for any	realization  of Algorithm \ref{alg2} if $\mathcal E_1$ happens.
			
On the other hand, by Lemma \ref{lemphi} and multivariate Chebyshev's inequality 		\cite{Chebyshevproof}, for any $s>0$ and $j\ge j_0$,
\begin{align}\label{Chev}
\mathbb{P}\left(\left\|\nabla \phi_{S_{k,j}}(x_{k,j};\delta_{k,j_0})-\mathbb{E}_u
				[\nabla \phi_{S_{k,j}}(x_{k,j};\delta_{k,j_0})] 	
				\right\|\leq s\right)\ge 1- \frac{\kappa_{\phi_k,2}\gamma_{k,j}^2}{s^2}.
\end{align}
For any $0<p<1$, let $s=\sqrt{\frac{ \kappa_{\phi_k,2}}{1-p}}	\gamma_{k,j}$.
Then  by (\ref{phi1bound}) and \eqref{Chev},
\begin{align}\label{pgrad}
&	\mathbb{P}\left[	\left\|\nabla\phi_{S_{k,j}}(X_{k,j};\Delta_{k,j_0})-\nabla
\phi_{k}(X_{k,j};\Delta_{k,j_0})\right\|\leq \left(\kappa_{\phi_k,1}+ \sqrt{\frac{
\kappa_{\phi_k,2}}{1-p}}\right)\Gamma_{k,j} \ \bigg |\   \mathcal{F}_{k,j}, \mathcal E_1 \right] \nonumber\\
&= 	\mathbb{P}\left[	\left\|\nabla	\phi_{S_{k,j}}(x_{k,j};\delta_{k,j_0})-\nabla\phi_{k}(x_{k,j};\delta_{k,j_0})
	\right\|\leq \left(\kappa_{\phi_k,1}+ \sqrt{\frac{	\kappa_{\phi_k,2}}{1-p}}
				\right)\gamma_{k,j}  \right] \nonumber\\
&\geq \mathbb{P}\left(\left\|\nabla \phi_{S_{k,j}}(x_{k,j};\delta_{k,j_0})-\mathbb{E}_u
				[\nabla \phi_{S_{k,j}}(x_{k,j};\delta_{k,j_0})] 	
				\right\|\leq \sqrt{\frac{ \kappa_{\phi_k,2}}{1-p}}	\gamma_{k,j}\right)\nonumber\\
&\ge p
\end{align}
for all $j\ge j_0$.
			%Denote $C_{\phi_k,1}+ \sqrt{\frac{ C_{\phi_k,2}}{1-p}}$

It follows from (\ref{limconverg}) and \eqref{gamma} that  there exists $j_1\ge j_0$  such that
\begin{align}\label{t6}
\|\nabla \phi_{k}(x_{k,j_1+1};\delta_{k,j_0})\|\leq \theta\|\nabla_x L_{s_{k}}
				(x_k,y_k) \|+ \frac{1}{2}\epsilon_k- \left(\kappa_{\phi_k,1}+
				\sqrt{\frac{\kappa_{\phi_k,2}}{1-p}}\right)\gamma_{k,j}.
\end{align}
Thus by \eqref{pgrad} and (\ref{t6}),
\begin{align}\label{t7}
\left\|\nabla\phi_{S_{k,j_1+1}}(x_{k,j_1+1};\delta_{k,j_0})\right\|\leq
				\theta\|\nabla_x L_{s_{k}}		(x_k,y_k) \|+ \frac{1}{2}\epsilon_k
\end{align}
holds with probability at least $p$.
			
Note that
$$
\nabla\phi_{s_{k,j+1}}(x_{k,j+1};\delta_{k,j})=\nabla_xL_{s_{k,j+1}}
(x_{k,j+1},y_k-\delta_{k,j}^{-1}c(x_{k,j+1}))
$$
and $\delta_{k,j}$ does not decrease for $j\geq j_0$.
By  the line 8 of Algorithm \ref{alg2}, if \eqref{t7} happens, it holds
\begin{align}
\|c(x_{k,j_1+1})\|\leq \theta	\|c(x_{k})\|+\frac{1}{2}\epsilon_k.
\end{align}
Since $z_{k,j_1+1}=r(x_{k,j_1+1})$, $y_{k,j_1+1}= y_k-\delta_{k,j_0}^{-1}c(x_{k,j_1+1})$ due to line 9 of Algorithm \ref{alg2}, by (\ref{t7}),  we have
			\begin{align*}
				\|F_{S_{k,j_1+1}}\|_*&=\left\|\nabla
				\phi_{k,S_{k,j_1+1}}(x_{k,j_1+1};\delta)\right\| +
				\|z_{k,j_1+1}-r(x_{k,j_1+1})\| + \|c(x_{k,j_1+1})\| \\
				&\leq \theta \left( \|\nabla_x 	L_{s_{k}}(x_k,y_k) \| +
				\|z_{k}-r(x_{k})\| + \|c(x_{k})\|   \right)  + \epsilon_k  \\
				&=\theta \|F_{s_k}\|_*+\epsilon_k
			\end{align*}
			holds with   probability at least  $p$.			
		\end{proof}

Now we consider the case that $\Delta_{k,j}$ decrease infinitely many times.
It is shown that Algorithm \ref{alg2} converges to a stationary point of the merit function of the constraints with probability one. It happens in numerical experiments.

\begin{lemma}\label{lemBC}
 Suppose that Assumptions \ref{ass1rd}(i) and \ref{ass2rd}(i)-(iii)  hold for each realization of 	Algorithm \ref{alg2} with fix $k$.
 If $\Delta_{k,j}$ decrease infinitely many	times, then the random sequence $\{X_{k,j}\}$
generated  by Algorithm \ref{alg2} 	satisfies
\begin{align}\label{feasible}
		\liminf_{j\to \infty}\|J^c(X_{k,j})^T c(X_{k,j})\|=0
\end{align}
almost surely.
\end{lemma}
		
\begin{proof}
Denote
$$
\mathcal E_2=\{\Delta_{k,j}\  \text{decrease infinitely many times}\}.
$$
If $\mathcal E_2 $  happens  for the realization of  Algorithm \ref{alg2} for some $k$,
then the set
$$
\mathcal{J}:= \{ j\in \mathbb{N}: \delta_{k,j}<\delta_{k,j-1} \}
$$
is infinite. Due to line 8 of Algorithm \ref{alg2}, for all $j\in\mathcal{J}$,
\begin{align}\label{tb1}
&\|(J^r_{s_{k,j}})^Tr(x_{k,j})-(J^c_{s_{k,j}})^Ty_k	+\delta_{k,j-1}^{-1}(J^c_{s_{k,j}})^Tc(x_{k,j})\|\nonumber\\
&=\|\nabla_xL_{s_{k,j}}(x_{k,j},y_k-\delta_{k,j-1}^{-1}c(x_{k,j}))  \|\nonumber\\
&\leq \theta	\|\nabla_x L_{s_{k}} (x_k,y_k) \|+	\frac{1}{2}\epsilon_k,
\end{align}
and	$\| c(x_{k,j}) \| > \theta\|c(x_{k})\|	+\frac{1}{2}\epsilon_k>0 $.
It then follows from Assumptions %\ref{ass1rd} and
\ref{ass2rd}(i)-(iii),
\eqref{tb1} and $\delta_{k,j} \ge \frac{1}{10}\delta_{k,j-1}$  that for all
$j\in\mathcal{J}$,
\begin{align*}
\frac{1}{10}\delta_{k,j}^{-1} \|(J^c_{s_{k,j}})^Tc(x_{k,j})\|
&\leq \delta_{k,j-1}^{-1}\|(J^c_{s_{k,j}})^Tc(x_{k,j})\|\nonumber\\
&\leq\|(J^r_{s_{k,j}})^Tr(x_{k,j})-(J^c_{s_{k,j}})^Ty_k\|
				+\theta\|\nabla_x L_{s_{k}} (x_k,y_k) \|+\frac{1}{2}\epsilon_k\nonumber\\
&<+\infty.
\end{align*}
Since $\delta_{k,j} = \delta_{k,j-1}$  for all $j\notin\mathcal{J}$,	we have
 $ \{(J^c_{s_{k,j}})^Tc(x_{k,j})\}_{j\in \mathcal{J}}	\to	0$.
			
Next we prove $\{\gamma_{k,j}\}_{j\in \mathcal{J}} \to 0$.
Otherwise there exists  $\nu>0$ such that the set $ \mathcal{J}_1=\{j\in \mathcal{J}:\gamma_{k,j}>\nu\}$ is
infinite. Hence the set $ \mathcal{J}_2=\{j\in \mathcal{J}:\gamma_{k,j}>\frac{\nu}{2}\}$ is also infinite.
Define $ \mathcal{J}_3=\{j\in \mathcal{J}_2:\gamma_{k,j+1}>\gamma_{k,j}\}$.
We show $\mathcal{J}_3$ is also infinite.
Otherwise,  there exists $j_0\in \mathcal{J}_2 $ such that $\gamma_{k,j+1}\leq \gamma_{k,j}$ for all $j\ge j_0$ in $\mathcal{J}_2$.
Then there exists $j_1\in\mathcal{J}\setminus \mathcal{J}_2, j_1>j_0$,
such that $\gamma_{k,j_1} \leq \frac{\nu}{2}$ and $\gamma_{k,j+1}\leq \gamma_{k,j}$ for all $j\in [j_0,j_1)\cap \mathcal{J}_2$.
Since $\mathcal{J}_2$ is infinite, by the updating rule	(\ref{gamaupd}), there exists $j_2\in\mathcal{J}_2, j_2>j_1$ such that $\frac{\nu}{2}<\gamma_{k,j_2}\leq \nu$ and $\gamma_{k,j}\leq \frac{\nu}{2}$
for all $j \in [j_1,j_2)\cap \mathcal{J}\setminus \mathcal{J}_2$.
By induction, it can be deduced  that $\gamma_{k,j} \leq \nu$ for all $j\in (j_2,+\infty)\cap \mathcal{J}$, which contradicts the infiniteness of $\mathcal{J}_1$. Thus, $\mathcal{J}_3$ is infinite.
Since $ \{(J^c_{s_{k,j}})^Tc(x_{k,j})\}_{j\in \mathcal{J}} \to 0$, there exists
$j_{\nu}\in\mathcal{J}_3\subset \mathcal{J}_2$, such that
\begin{align*}%\label{tb3}
\|(J^c_{s_{k,j_{\nu}}})^Tc(x_{k,j_{\nu}})\| < \frac{p_3\nu}{2}< p_3 \gamma_{k,j_{\nu}}, \quad					\gamma_{k,j_{\nu}+1}>\gamma_{k,j_{\nu}}.
\end{align*}
This contradicts (\ref{gamaupd}).
Thus, $\{\gamma_{k,j}\}_{j\in \mathcal{J}} \to 0$.
Besides, by Algorithm \ref{alg2}, $\gamma_{k,j+1}\leq\frac{1}{2}\gamma_{k,j}$ for $j\notin \mathcal{J}$,
so we have
\begin{align}\label{gammato}
\{\gamma_{k,j}\}_{j} \to 0.
\end{align}

By Lemma \ref{lemAB} and multivariate Chebyshev's inequality \cite{Chebyshevproof}, for any $s>0$,
\begin{align}\label{temp}
\mathbb{P}\left(	\left\|	J^c_{S_{k,j}}(x_{k,j})-\mathbb{E}_u [J^c_{S_{k,j}}(x_{k,j})] 	
				\right\|\leq s\right)\ge 1- \frac{d^2m n\kappa_{lgc}^2}{4s^2}\gamma_{k,j}^2.
\end{align}
By Assumption \ref{ass2rd}(ii) and (\ref{lem1B}),
\begin{align}\label{temp1}
 &	\left\| J^c_{S_{k,j}}(x_{k,j})^Tc(x_{k,j})-J^c(x_{k,j})^Tc(x_{k,j}) \right\| \leq
			 	\left\| J^c_{S_{k,j}}(x_{k,j})-J^c(x_{k,j})  \right\|\left\|  c(x_{k,j}) \right\|\nonumber\\
			 &\leq \left( \left\|  J^c_{S_{k,j}}(x_{k,j})-\mathbb{E}_u
			 [J^c_{S_{k,j}}(x_{k,j})]\right\|+\left\|\mathbb{E}_u
			 [J^c_{S_{k,j}}(x_{k,j})]- J^c(x_{k,j})\right\|  \right) \left\| c(x_{k,j}) \right\|\nonumber\\
			 & \leq \kappa_{bc}\left( \left\| J^c_{S_{k,j}}(x_{k,j})-\mathbb{E}_u [J^c_{S_{k,j}}(x_{k,j})]\right\|+d\sqrt{m}\kappa_{lgc}\gamma_{k,j} \right) .
\end{align}
Let $s=\sqrt{\frac{d^2m n\kappa_{lgc}^2}{2}} \gamma_{k,j}$ in (\ref{temp}).
Combining (\ref{temp}), (\ref{temp1}) and $\mathbb{E}_{u}[J^c_{S_{k,j}}(x_{k,j})]=			\mathbb{E}[J^c_{S_{k,j}}(X_{k,j})\mid \mathcal{F}_{k,j}]$, we obtain
\begin{align*}
&	\mathbb{P}\left(	\left\|
J^c_{S_{k,j}}(X_{k,j})^Tc(X_{k,j})-J^c(X_{k,j})^Tc(X_{k,j})
\right\|\leq \kappa_{bc}d\sqrt{m}\kappa_{lgc}\left(1+\sqrt{\frac{n}{2}}\right)\Gamma_{k,j}	\mid \mathcal{F}_{k,j} \right)\\
&= 	\mathbb{P}\left(	\left\|	J^c_{S_{k,j}}(x_{k,j})^Tc(x_{k,j})-J^c(x_{k,j})^Tc(x_{k,j})
				\right\|\leq \kappa_{bc}d\sqrt{m}\kappa_{lgc}\left(1+\sqrt{\frac{n}{2}}\right)\gamma_{k,j}	\right)\nonumber\\
&\ge \mathbb{P}\left(	\left\|	J^c_{S_{k,j}}(x_{k,j})-\mathbb{E}_u [J^c_{S_{k,j}}(x_{k,j})] 	
				\right\|\leq \sqrt{\frac{d^2m n\kappa_{lgc}^2}{2}} \gamma_{k,j}\right)\nonumber\\
&\ge \frac{1}{2}.
 \end{align*}
			
Let $\hat{\kappa}:= \kappa_{bc}d\sqrt{m}\kappa_{lgc}\left(1+\sqrt{\frac{n}{2}}\right)$.
Define the event
$$V_j=\left\lbrace 	\left\|	J^c_{S_{k,j}}(X_{k,j})^Tc(X_{k,j})-J^c(X_{k,j})^Tc(X_{k,j})
		\right\|\leq \hat{\kappa}\Gamma_{k,j}	\right\rbrace.
$$
Then $\mathbb{P}\left( V_j \mid \mathcal{F}_{k,j}\right)\ge \frac{1}{2}$.
Define a random walk
			\begin{align}\label{radwk}
				W^R_l=\sum_{j=0}^{l} (2 \textbf{1}_{V_j}-1).
			\end{align}
Obviously,
			\begin{align*}
				W^R_{l+1}=
				\begin{cases}
					W^R_{l}+1, & \text{if}\ \textbf{1}_{V_{l+1}}=1, \\
					W^R_{l}-1, & \text{if}\ \textbf{1}_{V_{l+1}}=0.
				\end{cases}
			\end{align*}
Let $\mathcal{F}^V_{l-1}$ be the $\sigma$-algebra generated by 		$\textbf{1}_{V_{0}},\ldots,\textbf{1}_{V_{l-1}}$, which is contained in $\mathcal{F}_{k,l}$,
Then
\begin{align*}
\mathbb E\left[ W^R_l \mid 		\mathcal{F}^V_{l-1}\right]&= \mathbb E
				\left[W^R_{l-1}  \mid\mathcal{F}^V_{l-1} \right] +\mathbb  E \left[	2\cdot
				\textbf{1}_{V_l} -1  \mid	\mathcal{F}^V_{l-1}\right] \nonumber\\
				&=  W^R_{l-1} + 2 \mathbb P \left( V_l \mid	\mathcal{F}^V_{l-1}\right) -1 \nonumber\\
				& \ge  W^R_{l-1},
\end{align*}
So $W^R_l$ is a submartingale with $\pm 1$ increments.
Thus, by \cite[Theorem 4.3.1]{DR2010},
			\begin{equation}\label{prow1}
				\mathbb P \left( \limsup_{l\to\infty} W^R_l=+ \infty
				\right) =1.
			\end{equation}
			
Suppose that \eqref{gammato} is not true.
We discuss the case that
\begin{align}\label{bigthe}
\|J^c(x_{k,j})^Tc(x_{k,j})\|> \varepsilon
\end{align}
happens for some $\varepsilon>0$ and all $j$.
Define the random variable
$$
\Pi_j=\ln_2 (\Gamma_{k,j}),
$$
whose realization is denoted as $\pi_j=\ln_2 (\gamma_{k,j} )$.
Since $\{\gamma_{k,j}\}_{j} \to 0$, there exists $j_0$ such that for all $j \ge j_0$,
\begin{align}\label{tb4}		
				\gamma_{k,j}\leq \min\left\lbrace	
			\frac{\varepsilon}{p_4+\hat{\kappa}},
			\frac{\gamma_{k}}{2} \right\rbrace.
\end{align}
We discuss two	cases for $j\in\mathcal{J}, j\ge j_0$.
			
If $\textbf{1}_{V_j}=1$, then
\begin{align*}
&\left\|J^c(x_{k,j})^Tc(x_{k,j}) \right\|-\left\|J^c_{s_{k,j}}(x_{k,j})^Tc(x_{k,j})\right\|\\
&\leq	\left\|	J^c_{s_{k,j}}(x_{k,j})^Tc(x_{k,j})-J^c(x_{k,j})^Tc(x_{k,j})	\right\|
\leq \hat{\kappa}\gamma_{k,j},
\end{align*}
which, together with (\ref{bigthe}), implies that
\begin{align*}
	\left\|	J^c_{s_{k,j}}(x_{k,j})^Tc(x_{k,j})\right\|\geq 	\varepsilon- \hat{\kappa}\gamma_{k,j}.
\end{align*}
Thus by (\ref{tb4}) and the updating rule (\ref{gamaupd}), we have $\gamma_{k,j+1}=2\gamma_{k,j} $.
 This gives
\begin{align}\label{tb5}
				\pi_{j+1}-\pi_j=1.
\end{align}
			
If $\textbf{1}_{V_j}=0$, then $\gamma_{k,j+1}\ge \frac{1}{2}\gamma_{k,j}$ due to (\ref{gamaupd}) and $j\in\mathcal{J}$.
Hence $\pi_{j+1}\ge \pi_j-1$. This  implies
			\begin{align}\label{tb6}
				\pi_{j+1}-\pi_j\ge -1.
			\end{align}
			
Combining (\ref{tb5}) and (\ref{tb6}), we have
			\begin{equation*}
				\pi_j-\pi_{j_0}\ge w^R_j-w^R_{j_0},
			\end{equation*}
where $w^R_j$ is a realization of $W^R_j$ corresponding to the particular realization $\pi_j$.
Because $\pi_{j_0}$ and $w^R_{j_0}$ is finite, by (\ref{prow1}), we have
	\begin{equation}\label{zproin}
				\mathbb{P} \left( \limsup_{j\to\infty}
				\Pi_j=+	\infty  \mid \mathcal E_2\right) =1.
	\end{equation}
On the other hand, by \eqref{gammato}, we have $\pi_j<0$ for sufficiently large $j$, which contradicts (\ref{zproin}).
Therefore, we have
			\begin{align*}
				\P\left( 	\liminf_{j\to \infty}\|J^c(X_{k,j})^T
				c(X_{k,j})\|=0 \mid \mathcal E_2 \right)  =1.
			\end{align*}
The proof is completed.
	\end{proof}
	
Combining Lemmas \ref{lembeforthem} and \ref{lemBC}, we obtain the convergence result of Algorithm \ref{alg2} as follows.

\begin{theorem}\label{theminner}
Suppose that Assumptions \ref{ass1rd} and \ref{ass2rd}(i)-(iii) hold for every realization of  Algorithm \ref{alg2} with $k$.
It holds that either
	\begin{itemize}
		\item[(i)] for any $\epsilon>0$ and $0<p<1$, we can find $j$ such that
		\begin{align}\label{akkt}
				\|F_{S_{k,j+1}}\|_* \leq \theta \|F_{s_k}\|_*+\epsilon_k
		\end{align}
with probability at least $p$; or	
		\item[(ii)] the random sequence $\{X_{k,j}\}$ generated  by Algorithm
		\ref{alg2}
		satisfies
		\begin{align*}
			\liminf_{j\to \infty}\|J^c(X_{k,j})^T c(X_{k,j})\|=0
		\end{align*}
		almost surely.
	\end{itemize}
\end{theorem}

%Lemma \ref{lembeforthem} helps us to prove that the merit function converges
%to zero if $\Delta_{k,j}$ decrease finitely many times.

Now we show that Algorithm \ref{alg1} converges	to an approximate KKT point
with arbitrary precision with high probability if Algorithm \ref{alg2} always
return to Algorithm
	\ref{alg1} with (\ref{innerstoprd}) being satisfied.

\begin{lemma}\label{Gsto0ther}
Suppose that Assumptions \ref{ass2rd}(v) holds for every realization of  Algorithm  \ref{alg1}, and Algorithm \ref{alg2} always achieves (\ref{innerstoprd}),
then $\{w_k\}$ generated by Algorithm \ref{alg1}
	satisfies
	\begin{align}\label{Gsto0}
		\lim_{k\to\infty}	F_{s_k}= 0.
	\end{align}
\end{lemma}

\begin{proof}
Let $\mathcal{G}:=\limsup\limits_{k\to\infty}\|F_{s_{k}}\|_*$ and $\epsilon_M=\sup_k \epsilon_k$.	
By (\ref{innerstoprd}),
\begin{align*}
\|F_{s_{k+1}}\|_*-\frac{\epsilon_M}{1-\theta}%&\leq \theta \left(			\|F_{s_{k}}\|_*-\frac{\epsilon_M}{1-\theta}	\right) +(\epsilon_k-\epsilon_M) \\
			&\leq \theta \left(\|F_{s_{k}}\|_*-\frac{\epsilon_M}{1-\theta}\right).
\end{align*}
This implies $\mathcal{G} \leq\frac{\epsilon_M}{1-\theta}$ because of $0<\theta<1$.
Taking the limit superior in (\ref{innerstoprd}), we have $\mathcal{G}\leq \theta \mathcal{G}$, which gives $\mathcal{G}=0$.
Hence,
\begin{align*}
0\leq \liminf\limits_{k\to\infty}\|F_{s_{k}}\|_* \leq \mathcal{G}=0,
\end{align*}
thus $\lim\limits_{k\to\infty}\|F_{s_{k}}\|_* =0$. So \eqref{Gsto0} holds true.
\end{proof}

%{\color{blue}
%Based on the above lemma, we can show that Algorithm \ref{alg1} converges to
%an approximate KKT point with arbitrary precision	and high probability.
%%if
%%$\Delta_{k,j}$ decrease finitely many times.
%}

	Note that \eqref{prob} is equivalent to
\begin{align}\label{prob2}
	\min_{x\in\mathbb{R}^n} \ \frac{1}{2}\|z\|^2,\quad
	\text{s.t.} \quad r(x)=z, \ c(x)=0.
\end{align}
The KKT conditions for \eqref{prob2} are
\begin{align}\label{KKT-prob2}
	F(w)=	\begin{bmatrix}
		J^r(x)^Tz-J^c(x)^Ty\\
		r(x)-z\\
		c(x)
	\end{bmatrix}=0.
\end{align}
where $y$ and $-z$ are the Lagrange multipliers for $c(x)=0$ and
$r(x)-z=0$, respectively.
Based on Lemma \ref{Gsto0ther}, we have:

\begin{theorem}\label{outcon}
Suppose that Assumptions \ref{ass1rd}(i) and \ref{ass2rd}(iii)-(v) hold for every realization of  Algorithm  \ref{alg1}, and Algorithm \ref{alg2} always achieves (\ref{innerstoprd}),
then for any $\epsilon>0$ and $0<p<1$, we can find $k_0$ such that
\begin{align}\label{akkt}
		\|F_{k_0}\|  \leq \epsilon
\end{align}
with probability at least $p$.
\end{theorem}

\begin{proof}
	Denote $J^r_{k}:=J^r(x_k), J^c_{k}:=J^c(x_k), F_{s_k}:=F_{s_k}(w_k;
	x_k,y_k,\rho_k,\delta_k), F_k=F(w_k)$.
  By (\ref{Grd}) and \eqref{KKT-prob2},
	\begin{align*}
		F_{s_k}-F_k=\begin{bmatrix}
			\left(J^r_{s_k}- J^r_{k}\right) ^Tz_k-
			\left(J^c_{s_k}- J^c_{k}\right) ^Ty_k\\
			0\\
			0
		\end{bmatrix}.
	\end{align*}
	Hence,
	\begin{align}\label{tet1}
		\|F_{s_k}-F_k\| \leq \|J^r_{s_k}- J^r_{k}\|\|z_k\|+	\|J^c_{s_k}-
		J^c_{k}\|\|y_k\|.
	\end{align}
By  Lemma \ref{lemAB} and multivariate 	Chebyshev's inequality
	\cite{Chebyshevproof},
	for any $s>0$,
	\begin{align*}
		\mathbb{P}\left[\left\|J^r_{S_k}(x_k)-\mathbb{E}_u[J^r_{S_k}(x_k)]\right\|\leq
		s	\right]\ge 1- \frac{d^2p n\kappa_{lgr}^2}{4 s^2} \gamma_k^2.
	\end{align*}
	For any $0<p<1$, let  $s=d\kappa_{lgr}\sqrt{\frac{p n}{4(1-p)}}\gamma_{k}$.
	By (\ref{lem1A}),
	\begin{align}
		&	\mathbb{P}\left[	\left\|J^r_{S_k}(X_{k})-J^r(X_k)
		\right\|\leq d\sqrt{p}\kappa_{lgr}\left(1+\sqrt{\frac{n}{4(1-p)}}
		\right)\Gamma_{k}\ \bigg |\ \mathcal{F}_{k}				\right]
		\nonumber\\
		&= 	\mathbb{P}\left[	\left\|J^r_{S_k}(x_{k})-J^r(x_k)
		\right\|\leq d\sqrt{p}\kappa_{lgr}\left(1+\sqrt{\frac{n}{4(1-p)}}
		\right)\gamma_{k}\right] \nonumber\\
		&\ge \mathbb{P}	
		\left[\left\|J^r_{S_k}(x_k)-\mathbb{E}_u[J^r_{S_k}(x_k)]\right\|\leq
		d\kappa_{lgr}\sqrt{\frac{pn}{4(1-p)}}\gamma_{k}\right]  \nonumber\\
		&\ge p.
	\end{align}
	Similarly, we have
	\begin{align}
		&	\mathbb{P}\left[	\left\|J^c_{S_k}(X_{k})-J^c(X_k)
		\right\|\leq d\sqrt{m}\kappa_{lgc}\left(1+ \sqrt{\frac{n}{4(1-p)}}
		\right)\Gamma_{k} \ \Big |\ \mathcal{F}_{k}	\right]
		\ge p.
	\end{align}
Let $\kappa_G:=(\kappa_{bz}d\sqrt{p}\kappa_{lgr}+\kappa_{by}
	d\sqrt{m}\kappa_{lgc} ) \left(1+\sqrt{\frac{n}{4(1-p)}}\right)$.
	It then follows from (\ref{tet1}) and Assumptions \ref{ass2rd}(iii) that
	\begin{align}\label{tet2}
		&	\mathbb{P}\left[	\left\|F_{S_k}(W_k)-F_k(W_k)
		\right\|\leq \kappa_G	\Gamma_{k} \mid \mathcal{F}_{k}
		\right]  =  \mathbb{P}\left[	\left\|F_{S_k}(w_k)-F_k
		\right\|\leq \kappa_G		\gamma_{k}
		\right]      \ge p.
	\end{align}
		By Assumption \ref{ass2rd}(iv), Lemma \ref{Gsto0ther} and
		(\ref{sysrd}), we
	have
	$\{d_{w_k}\}_k \to 0$.
	Note that at the $k$-th iteration, $\gamma_{k+1}=\|d_{x_k}\|$ if Algorithm
	\ref{alg2} is not called;
	otherwise we set $\gamma_{k+1}:=\gamma_{k,j_{t}+1}$,
	where $j_t$ is the index given by Algorithm \ref{alg2} such that
	(\ref{innerstoprd}) is satisfied.
	By  the updating rule of $\gamma_{k,j}$, we have $\gamma_{k+1} \leq
	\gamma_{k}$.
	Thus, $\{d_{w_k}\}_k \to 0$ implies $\{\gamma_{k}\}_k \to 0$.
	Hence by  Lemma 	\ref{Gsto0ther},
	for any $\epsilon>0$,
	\begin{align}\label{lst}
		\kappa_G \gamma_k \leq \frac{1}{2}\epsilon, \quad 	\|F_{s_k}\| \leq
		\frac{1}{2}\epsilon
	\end{align}
	holds for sufficiently large $k$ for every realization.
	Since $\|F_k\| \leq \|F_{s_k}\|+ \|F_{s_k}-F_k\| $,  by (\ref{tet2}) and
	(\ref{lst}),  we can find $k_0$ such that
	\begin{align*}
		\|F_{k_0}\|  \leq \epsilon
	\end{align*}
	with probability at least $p$.
\end{proof}

\section{Numerical experiments}\label{sec4}
In this section, we test Algorithm \ref{alg1} (DFRCNLS) and compare it with
two derivative-free algorithms for constrained optimization problems: COBYLA
(cf. \cite{PowellCOBYLA}) using	Matlab interface of PDFO\cite{pdfo}, a direct search optimization method
	that models the objective and constraint functions by linear interpolation,
and DEFT-FUNNEL (cf. \cite{Sampaio2021}), an  global optimization algorithm that belongs to the class of trust-region sequential quadratic optimization algorithms and uses the polynomial interpolation models as surrogates for the black-box functions.

The experiments are implemented on a laptop  with an AMD Core R7-6800H CPU (3.20GHz) and 32GB of RAM, using Matlab R2022b.

\subsection{Implementation details}
We set $\delta_0=1, \rho_0=0, \epsilon_{0}=10^{3}, \gamma_0=1, \theta=0.99$ in Algorithm \ref{alg1},
and $p_0=0.001, p_1=0.25, p_2=0.75, p_3=10^{-10}, p_4=10^{12}, \lambda_{\min}=10^{-8}$ in Algorithm \ref{alg2}, respectively.
We take $\delta_k=\max\left\{10^{-6}, \min\left\{0.1\delta_{k-1},\:\|F_{s_{k}}\|_*\right\}\right\}$ and
$\epsilon_{k+1}=\max\{\min\{10^{3}\delta_{k},0.99\epsilon_{k}\},0.9\epsilon_{k}\}$ (cf. \cite{Orban2020}).
The regularization parameter $\rho_{k}$ is updated as that in \cite[Algorithm 5.1]{Orban2020} so that the coefficient matrices of the system (\ref{sysrd}) are nonsingular.
The initialization multiplier $y_{0}$ is obtained by solving the problem
\begin{align*}	
\min\limits_{y}\:\frac{1}{2}
		\|(J^c_{s_0})^{T}y-(J^r_{s_0})^{T}r(x_{0})\|^{2}.
\end{align*}

Algorithm \ref{alg1} runs up to $K_1$ iterations and stops at iteration $k < K_1$ if
\begin{align}\label{nat}
\max\left\{\dfrac{\|(J^r_{s_{k}})^T	 		r(x_k)-(J^c_{s_{k}})^Ty_k\|_\infty}{\max\left\{100,\:\dfrac{\|y_k\|_1}{m}\right\}\Big /100},
	 	\:\|c(x_k)\|_\infty\right\}\leq 10^{-5}.
\end{align}
Algorithm \ref{alg2} runs up to $K_2$ iterations and stops at iteration $k < K_2$ if 		
 $\|(J^c_{s_{k,j}})^T c(x_{k,j})\|\leq 10^{-6}$, then it returns to Algorithm \ref{alg1}.
 The defaults values of $K_1$ and $K_2$ are 150 and 50, respectively.

In the implementation, the approximate Jacobian matrices are computed by \eqref{nabrrd}--\eqref{JBrd}. 	
We can generate $n$ independent random vectors $w_1,\ldots,w_n$ with identical distribution $\mathcal{N}(0,I)$ at each iteration, then compute the QR decomposition of the matrix  $[w_1,\ldots,w_n]$ to obtain $u_1,\ldots,u_n$; the corresponding algorithm is denoted as DFRCNLS-OSSv1.
We can also generate ten orthogonal direction sets, then choose one of them randomly as directions at each iteration; in this case the algorithm is denoted as DFRCNLS-OSSv2.
If we take $u_j=e_j(j=1,\ldots,n)$, then the Jacobian matrices are approximated by the finite difference, and the corresponding algorithm is denoted as DFRCNLS-FD,

We take the approximate Hessian matrices of $r_i(x)$, denoted by $H^{r_i}_{s}$, as zero matrix.
The approximate Hessian matrices of $c_i(x)$, denoted by $H^{c_i}_{s}$, are updated in the following three ways: by the symmetric rank-one update (SR1)
 \begin{align*}
 (H^{c_i}_{s})^+= \left\{\begin{array}{ll}
 H^{c_i}_{s}+ \dfrac{( y^{c_i}-H^{c_i}_{s}t)( y^{c_i}- H^{c_i}_{s}t)^T}{ ( y^{c_i}-H^{c_i}_{s}t)^Tt},
 & \mbox{if}\ \|( y^{c_i}-H^{c_i}_{s} t)^Tt \|\ge 10^{-7},\\
 H^{c_i}_{s}, & \mbox{otherwise},
 \end{array} \right.
 \end{align*}
 where $t=x^+ -x$, $y^{c_i}=[(J_{s}^{c})^+ ]_i^T-[J^c_{s}]_i^T $
and $[\cdot]_i$ represents the $i$-th row of the matrix,
or by the BFGS update
 \begin{align*}
(H^{c_i}_{s})^+= \left\{\begin{array}{ll}
  H^{c_i}_{s}+\dfrac{y^{c_i}(y^{c_i})^T}{t^Ty^{c_i}}- \dfrac{H^{c_i}_{s} t(H^{c_i}_{s} t)^T}{t^TH^{c_i}_{s} t},
 & \mbox{if}\  \|t^Ty^{c_i}\|\ge 10^{-7},\\
 H^{c_i}_{s}, & \mbox{otherwise},
 \end{array} \right.
 \end{align*}
or $(H^{c_i}_{s})^+=0$.

\subsection{Numerical results}
We present the numerical results using the 	performance profile developed by Dolan, Moré and Wild\cite{More2009}.
Denote by $\mathcal{S}$ the set of solvers and $\mathcal{P}$ the set of test problems, respectively.
For  each problem $p\in{\mathcal{P}}$, define a merit function
\begin{align}\label{metfc}
	\varphi(x)=\begin{cases}f(x),&\text{if}\:\|c(x)\|_{\infty}\leq1e-6,\\
		 f(x)+10^4  \|c(x)\|_{\infty},&\text{otherwise}.\end{cases}
	\end{align}
If a solver $s\in{\mathcal{S}}$  gives a point	$x$ that satisfies
\begin{align}\label{convtest}
\frac{	\varphi(x)-\varphi_p^* }{\varphi(x_0)-\varphi_p^*}\leq\tau,
\end{align}
where $\tau\in(0,1)$ is the tolerance and
$\varphi_p^*$ is the smallest value of merit function of problem $p$ obtained by any solver	in ${\mathcal{S}}$,
 we say that $s\in{\mathcal{S}}$  solves problem $p\in{\mathcal{P}}$ up to the	 convergence test.

Let $t_{p,s}$ be the least  number of function evaluations required by the solver  $s$ to solve the problem $p$ up to the convergence test.
If $s$ does not satisfy the convergence test for $p$ within the maximal number of function evaluations, then let $ t_{p,s}=\infty$.
The performance profile of $s$ is defined as
\begin{align}\label{ra}
		\pi_{s}(\alpha)=\frac{1}{n_{p}} \#\big\{p\in\mathcal{P}:r_{p,s}\leq\alpha\big\},
\end{align}
where  $r_{p,s}=\frac{t_{p,s}} 	{\min\{t_{p,s}:s\in\mathcal{S}\}}$ is the performance ratio,
$n_p$ is the number of test problems in $\mathcal{P}$ and $\#$ indicates the cardinal number of a set.
Hence,	$\pi_{s}(1)$ is the proportion of problems that $s$ performs better than any other solvers in $\mathcal{S}$ and $\pi_{s}(\alpha)$ is the proportion of problems solved by $s$ with a performance ratio at most $\alpha$. A better solver is indicated by a higher $\pi_{s}(\alpha)$.	
	
We collect 32 equality constrained nonlinear least squares problems from \cite{HS1981, HS1987, LL1999}
and	71 unconstrained ones from \cite{LL2018}.
We also create the degenerate problems by duplicating the constraints and squaring each function to the 32 equality constrained problems (cf. \cite{Orban2020}), that is, we solve
\begin{align*}
	\min_{x\in \mathbb R^n}& \quad\frac{1}{2}\sum_{i=1}^{p}r_i(x)^2\\
		\text{s.t.}& \quad c_i(x)=0,\quad  i=1,\ldots,m,\\
		&\quad c_i(x)^2=0, \quad i=1,\ldots,m.
\end{align*}
For each problem, we run DFRCNLS twenty times and take the median of the objective function value.

The results are given in the figures, where $\alpha$ is displayed in $\log_2$-scale.

 %\begin{figure}[H]
 \begin{figure}[htbp]
	\centering
	\begin{subfigure}{0.32\linewidth}
		\centering
		\includegraphics[width=\linewidth]
		{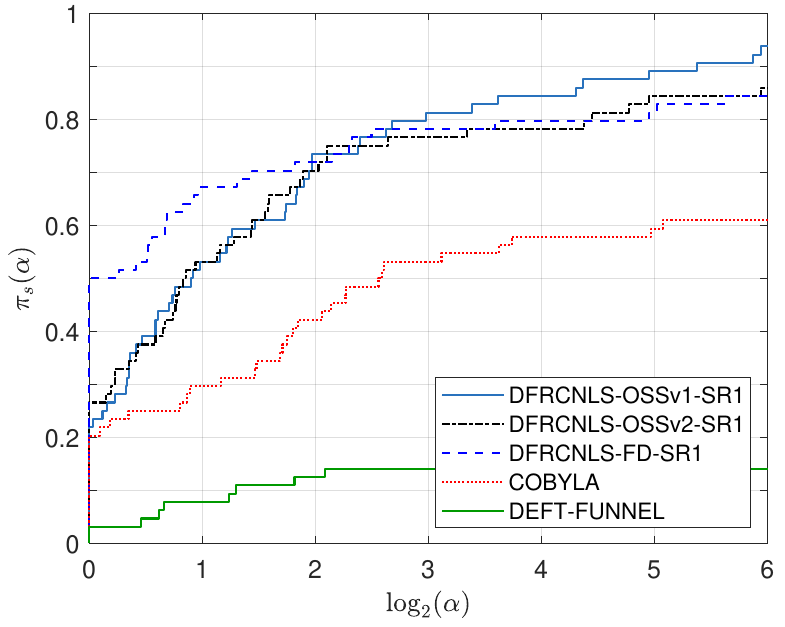}
	%	\caption{DFRCNLS-SR1 vs others }
		\label{f1}%文中引用该图片代号
	\end{subfigure}
	\centering
	\begin{subfigure}{0.32\linewidth}
		\centering
		\includegraphics[width=\linewidth]
			{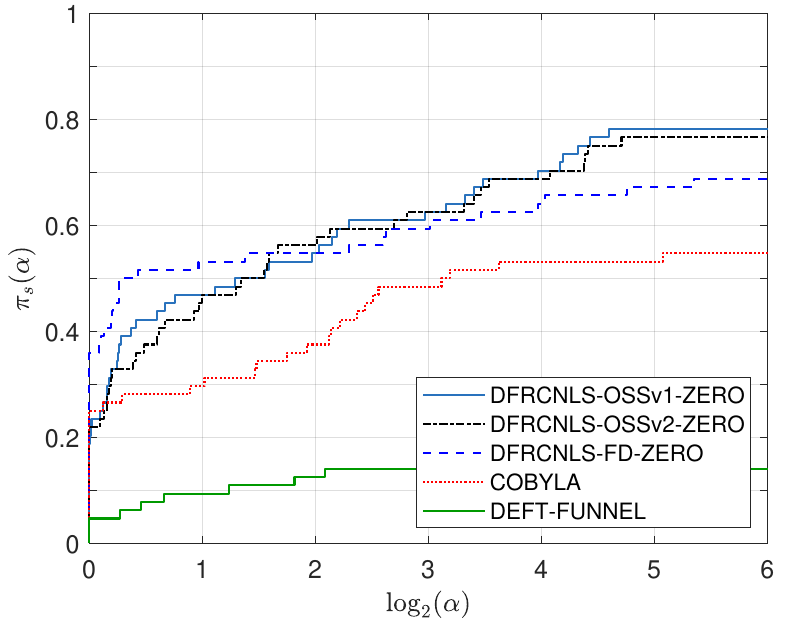}
	%	\caption{DFRCNLS-ZERO vs  others}
		\label{f2}%文中引用该图片代号
	\end{subfigure}
	\centering
	\begin{subfigure}{0.32\linewidth}
		\centering
		\includegraphics[width=\linewidth]
	{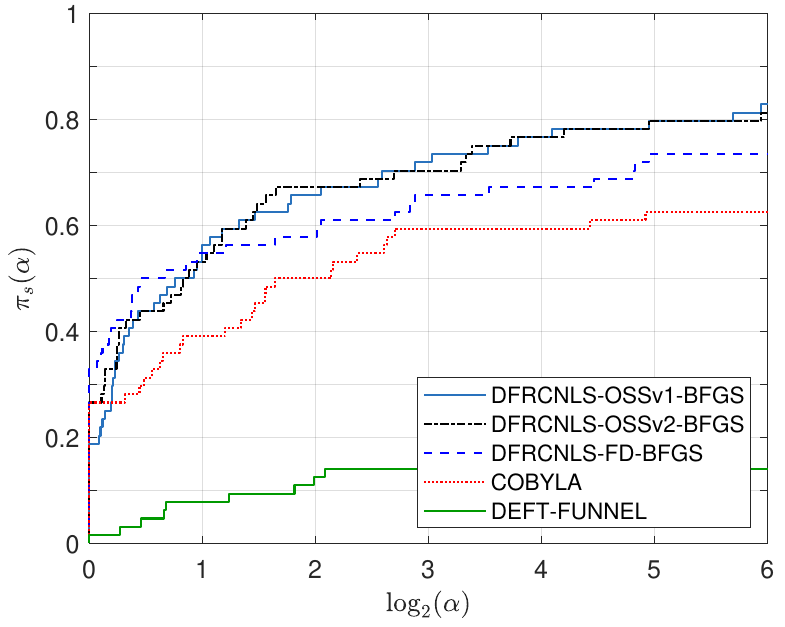}
	%	\caption{ DFRCNLS-BFGS vs  others}
		\label{f3}%文中引用该图片代号
	\end{subfigure}
   \caption{Performance profile on equality constrained nonlinear least squares problems using different Hessian approximations with tolerance $\tau=10^{-5}$.}
	\label{fcoby}
\end{figure}

 %\begin{figure}[H]
 \begin{figure}[htbp]
		\centering
		\begin{subfigure}{0.32\linewidth}
			\centering
			\includegraphics[width=\linewidth]
			{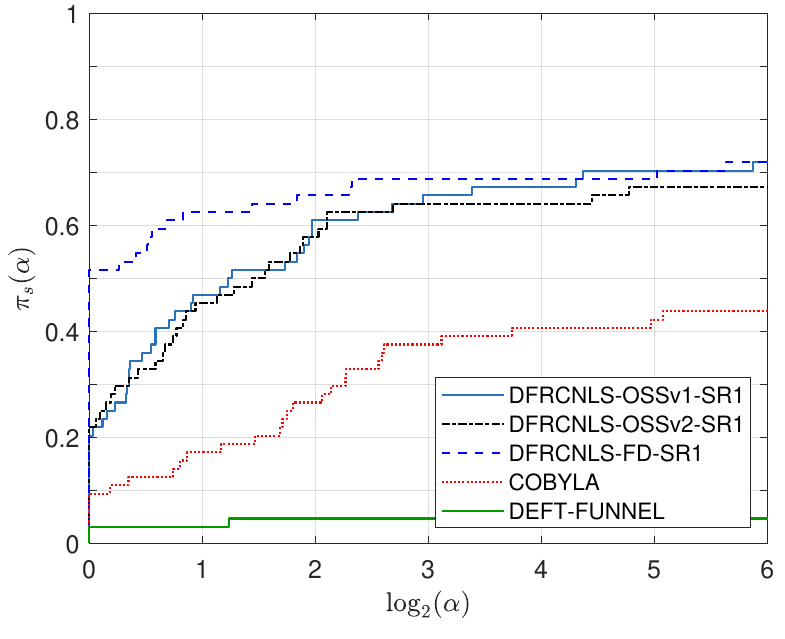}
	%		\caption{DFRCNLS-SR1 vs  others}
			\label{f11}%文中引用该图片代号
		\end{subfigure}
		\centering
		\begin{subfigure}{0.32\linewidth}
			\centering
			\includegraphics[width=\linewidth]
			{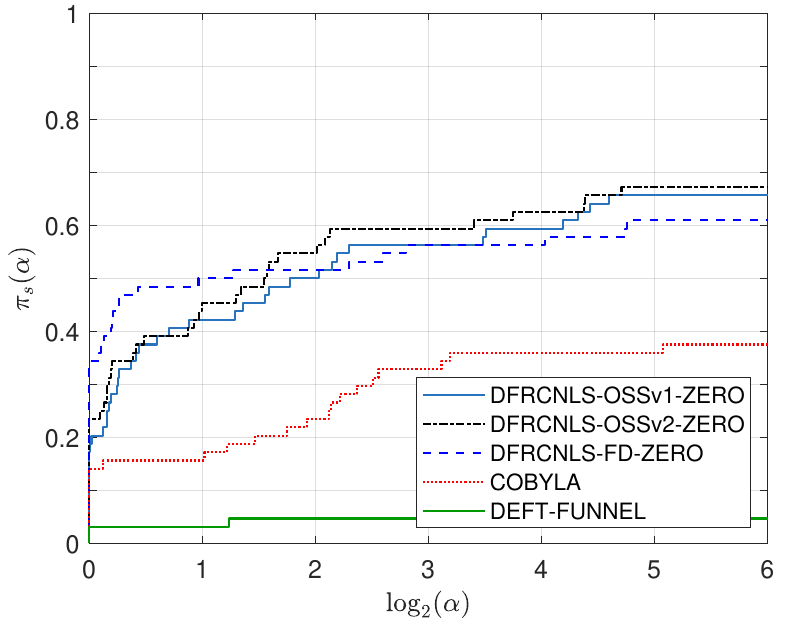}
	%		\caption{DFRCNLS-ZERO vs  others}
			\label{f21}%文中引用该图片代号
		\end{subfigure}
		\centering
		\begin{subfigure}{0.32\linewidth}
			\centering
			\includegraphics[width=\linewidth]
			{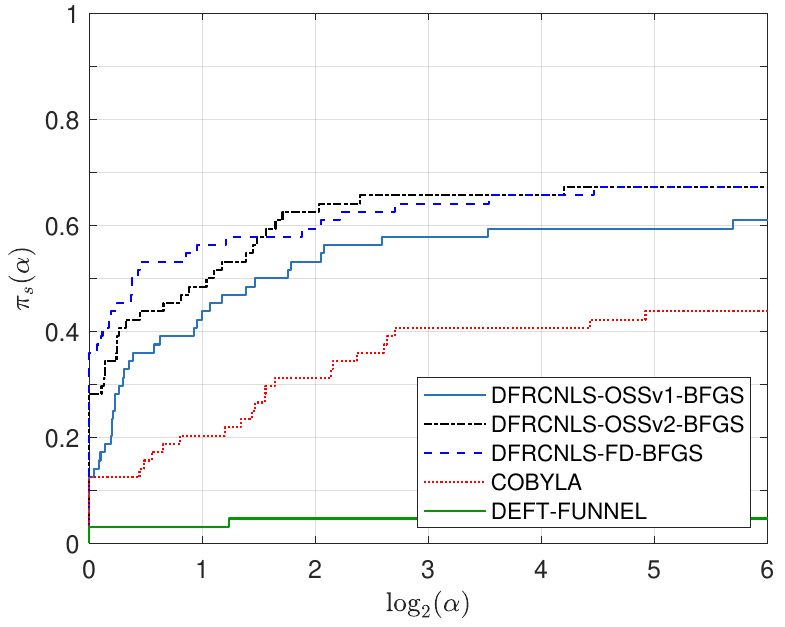}
	%		\caption{ DFRCNLS-BFGS vs  others}
			\label{f31}%文中引用该图片代号
		\end{subfigure}
		\caption{Performance profile on equality constrained nonlinear least squares problems using different Hessian approximations with tolerance $\tau=10^{-7}$.}
		\label{fcoby1}
	\end{figure}

As shown in Figure \ref{fcoby}, for equality constrained nonlinear least squares problems,
DFRCNLS-FD usually performs best among all the five solvers when $\tau=10^{-5}$,
and the SR1 updating of Hessian approximations of constraint functions seem more desirable than the BFGS updating and zero approximation.
At the higher accuracy $\tau=10^{-7}$, DFRCNLS still perform better than other solvers.

 %\begin{figure}[H]
 \begin{figure}[htbp]
	\centering
	\begin{subfigure}{0.32\linewidth}
		\centering
		\includegraphics[width=\linewidth]
		{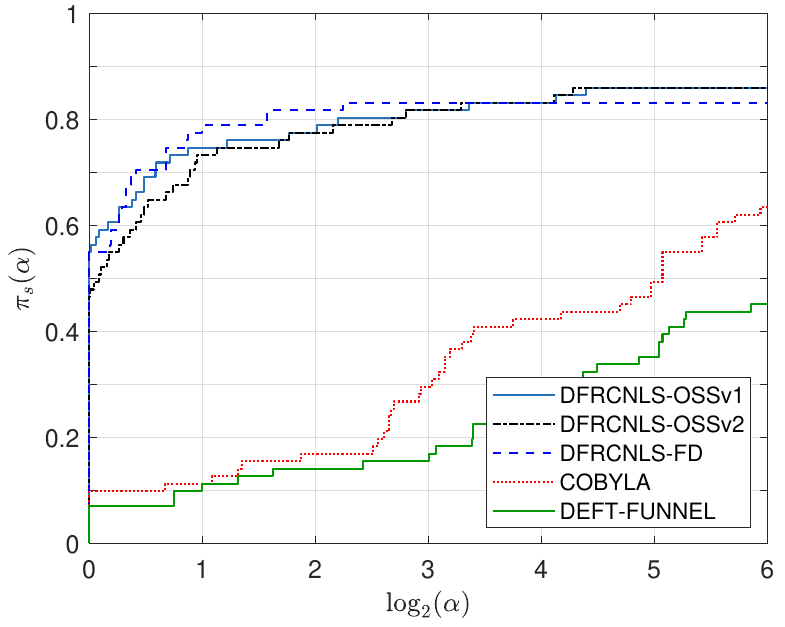}
		\caption{  Tolerance $\tau=10^{-5}$ }
		\label{f14}%文中引用该图片代号
	\end{subfigure}
	\centering
	\begin{subfigure}{0.32\linewidth}
		\centering
		\includegraphics[width=\linewidth]
			{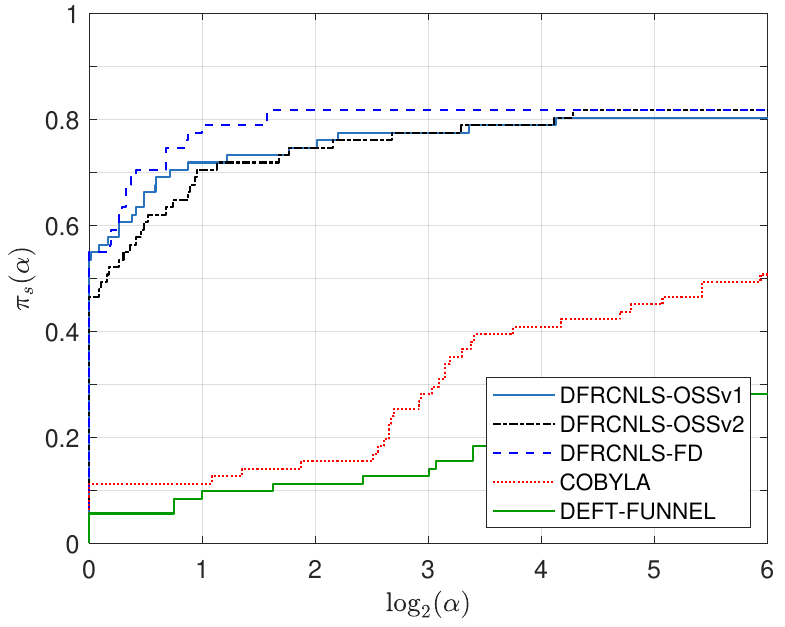}
		\caption{Tolerance $\tau=10^{-7}$}
		\label{f24}%文中引用该图片代号
	\end{subfigure}
%	\centering
%	\begin{subfigure}{0.32\linewidth}
%		\centering
%		\includegraphics[width=\linewidth]
%	{./unProblem/OUR_BFGS_mid_vs2method.perf_5.natural_cr.png}
%	%	\caption{ DFRCNLS-BFGS vs  others}
%		\label{f34}%文中引用该图片代号
%	\end{subfigure}
   \caption{Performance profile on unconstrained nonlinear
   least squares problems. }
	\label{fcoby4}
\end{figure}

Note that there are no Hessian approximations of constraint functions for unconstrained nonlinear least squares problems. As presented in Figure \ref{fcoby4}, the performances of three implementations of 	DFRCNLS are similar and they seem to have advantages over other methods for unconstrained problems.

 %\begin{figure}[H]
 \begin{figure}[htbp]
 	\centering
 	\begin{subfigure}{0.32\linewidth}
 		\centering
 		\includegraphics[width=\linewidth]
 		{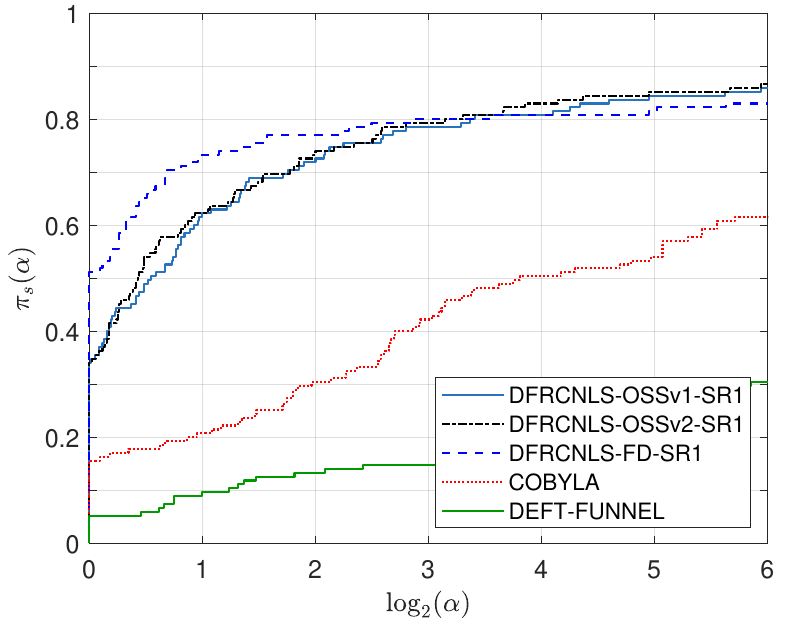}
 	%	\caption{DFRCNLS-SR1 vs  others}
 		\label{f12}%文中引用该图片代号
 	\end{subfigure}
 	\centering
 	\begin{subfigure}{0.32\linewidth}
 		\centering
 		\includegraphics[width=\linewidth]
 		{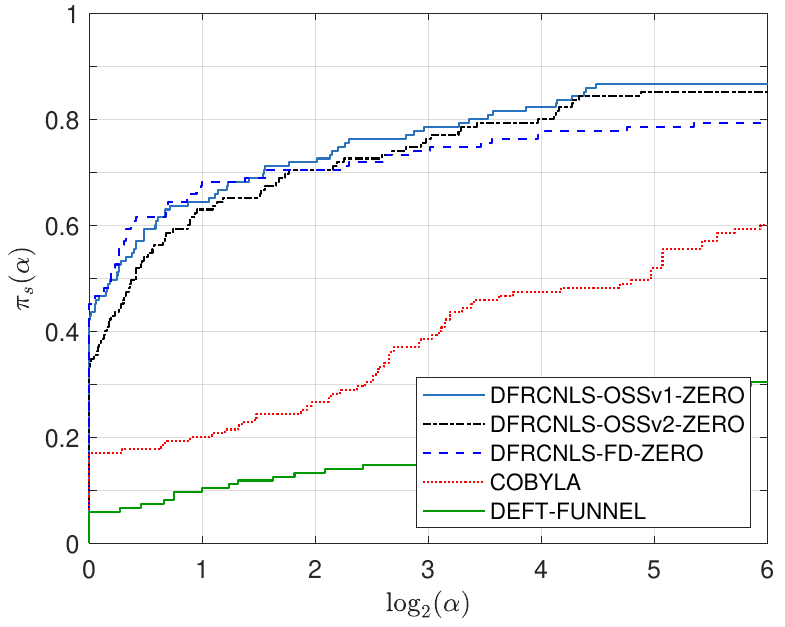}
 	%	\caption{DFRCNLS-ZERO vs  others}
 		\label{f22}%文中引用该图片代号
 	\end{subfigure}
 	\centering
 	\begin{subfigure}{0.32\linewidth}
 		\centering
 		\includegraphics[width=\linewidth]
 		{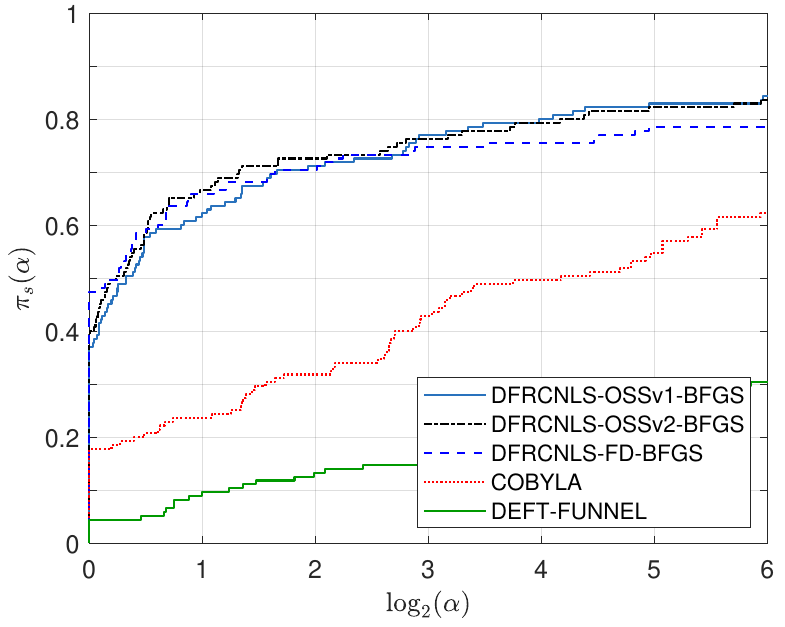}
 	%	\caption{ DFRCNLS-BFGS vs  others}
 		\label{f32}%文中引用该图片代号
 	\end{subfigure}
 	\caption{Performance profile on all 135 test problems using different Hessian approximations with tolerance $\tau=10^{-5}$.}
 	\label{fcoby2}
 \end{figure}

 %\begin{figure}[H]
 \begin{figure}[htbp]	
 \centering
 	\begin{subfigure}{0.32\linewidth}
 		\centering
 		\includegraphics[width=\linewidth]
 		{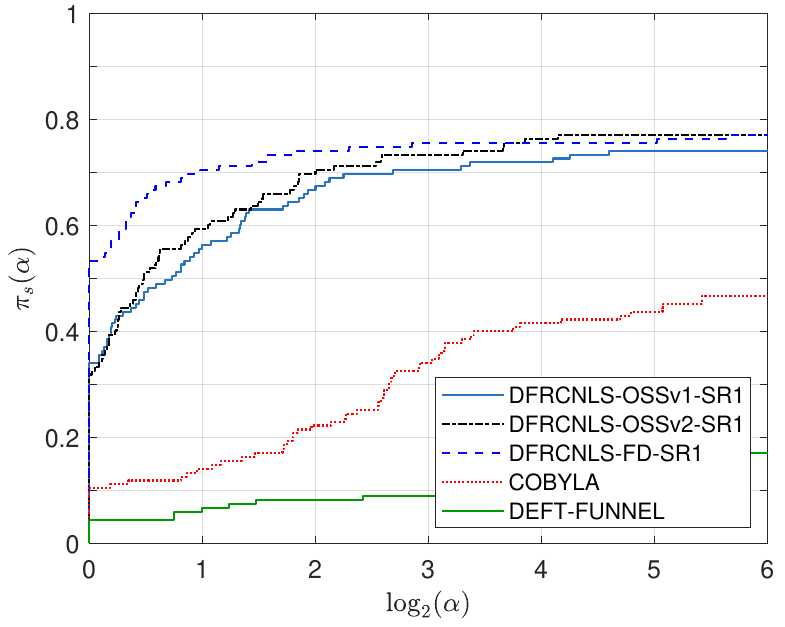}
 	%	\caption{DFRCNLS-SR1 vs  others}
 		\label{f13}%文中引用该图片代号
 	\end{subfigure}
 	\centering
 	\begin{subfigure}{0.32\linewidth}
 		\centering
 		\includegraphics[width=\linewidth]
 	{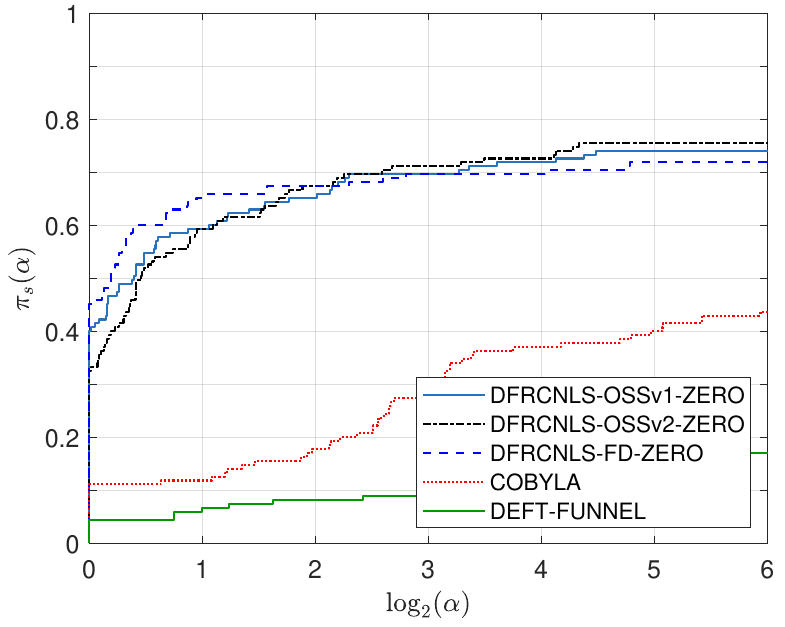}
 %		\caption{DFRCNLS-ZERO vs  others}
 		\label{f23}%文中引用该图片代号
 	\end{subfigure}
 	\centering
 	\begin{subfigure}{0.32\linewidth}
 		\centering
 		\includegraphics[width=\linewidth]
 		{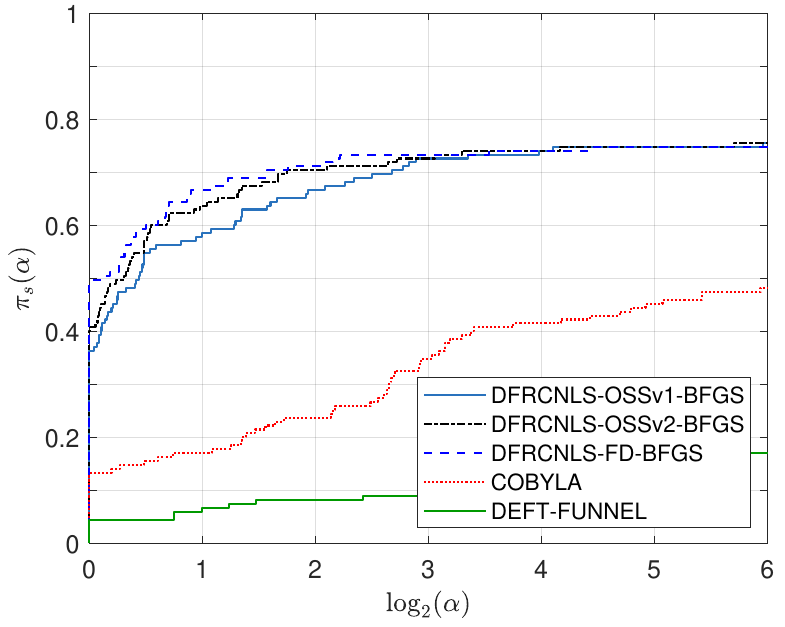}
 %		\caption{ DFRCNLS-BFGS vs  others}
 		\label{f33}%文中引用该图片代号
 	\end{subfigure}
 	\caption{Performance profile on all 135 test problems using different Hessian approximations with tolerance $\tau=10^{-7}$.}
 	\label{fcoby3}
 \end{figure}

As we can see from the figures, DFRCNLS is generally efficient than other solvers.
The reason may be that COBYLA and DEFT-FUNNEL are designed for general nonlinear constrained optimization problems while DFRCNLS is developed for equality constrained nonlinear least squares problems and makes use of the least squares structure.

%			\bibliographystyle{siam}
%		\bibliography{ref}

%\cite{Bortz98}

\end{document}